\documentclass[hidelinks,onefignum,onetabnum]{siamart251216}

\usepackage{comment}

\usepackage{amsmath,amssymb,subfigure}      
\usepackage{booktabs,multirow,threeparttable}  
\usepackage{lipsum}
\usepackage{amsfonts}
\usepackage{graphicx}
\usepackage{epstopdf}
\usepackage{algorithmic}
\usepackage{mathrsfs} 
\usepackage{multirow}

\ifpdf
  \DeclareGraphicsExtensions{.eps,.pdf,.png,.jpg}
\else
  \DeclareGraphicsExtensions{.eps}
\fi

\newsiamremark{remark}{Remark}
\newsiamremark{Assumption}{Assumption}
\newsiamremark{assumption}{Assumption}
\crefname{Assumption}{Assumption}{Hypotheses}
\newsiamthm{claim}{Claim}

\usepackage{amsopn}

\renewcommand{\div}{\operatorname{div}}
\graphicspath{{./figure/}{../figure/}}

\newcommand{\dx}{\,{\rm d}x}
\newcommand{\ds}{\,{\rm d}s}

\newcommand{\Oplus}{\ensuremath{\vcenter{\hbox{\scalebox{1.5}{$\oplus$}}}}}

\newcommand{\tr}{\operatorname{tr}}

\newcommand{\dev}{\operatorname{dev}}
\newcommand{\sym}{\operatorname{sym}}

\title{A Staggered Discontinuous Galerkin Method for \\
linear elasticity problem on Polytopal Meshes\thanks{Submitted to the editors DATE.
\funding{The first author was supported by DMS-2012465 and DMS-2309785. 
The third and the fourth author were supported by the National Key Research and Development Program of China (grant No. 2023YFA1008803), and the Key Laboratory of Symbolic Computation and Knowledge Engineering of Ministry of Education of China housed at Jilin University.}}}

\author{Long Chen\thanks{Department of Mathematics,
University of California at Irvine,
Irvine, CA 92697, USA
  (\email{chenlong@math.uci.edu}).}
\and Xuehai Huang\thanks{School of Mathematics, 
Shanghai University of Finance and Economics, 
Shanghai 200433, China  
  (\email{huang.xuehai@sufe.edu.cn}).}
\and Ruishu Wang\thanks{School of Mathematics,
Jilin University, 
Changchun, Jilin 130012, China
  (\email{wangrs\_math@jlu.edu.cn}).}
  \and Ran Zhang\thanks{School of Mathematics,
Jilin University, 
Changchun, Jilin 130012, China
  (\email{Zhangran@jlu.edu.cn}).}
}

\begin{document}

\maketitle

\begin{abstract}
This paper develops a novel staggered discontinuous Galerkin (SDG) method for linear elasticity based on the Hellinger-Reissner variational principle. We construct symmetric stress spaces with normal continuity across element boundaries on arbitrary polytopal meshes, while approximating the displacement field using piecewise polynomial functions defined on the same meshes. The method is locking-free and satisfies a local balance of linear momentum and angular momentum. We present a comprehensive theoretical analysis, including proofs of stability and error estimates. The formulation admits a hybridizable structure, which significantly simplifies the numerical implementation. Numerical experiments validate the theoretical results and demonstrate the effectiveness of the proposed approach.
\end{abstract}


\section{Introduction}
We develop a novel staggered discontinuous Galerkin (SDG) method for the linear elasticity problem. The Hellinger-Reissner mixed variational principle seeks the stress $\boldsymbol{\sigma}\in H(\div,\Omega;\mathbb S)$ and the displacement $\boldsymbol{u}\in L^2(\Omega;\mathbb{R}^d)$ such that
\begin{subequations}\label{equ_mixed_formulation}
\begin{align}
(\mathcal{A}\boldsymbol{\sigma},\boldsymbol{\tau}) + (\div\boldsymbol{\tau},\boldsymbol{u}) &= 0 \qquad\quad\;\;\, \forall~\boldsymbol{\tau}\in H(\div,\Omega; \mathbb S), \\
(\div\boldsymbol{\sigma},\boldsymbol{v}) &= -(\boldsymbol{f},\boldsymbol{v}) \quad \forall~\boldsymbol{v}\in L^2(\Omega; \mathbb{R}^d),
\end{align}
\end{subequations}
where the symmetric stress space is
$$ H(\div,\Omega; \mathbb S)=\{\boldsymbol{\tau}\in [L^2(\Omega)]^{d\times d}: \div \boldsymbol{\tau}\in L^2(\Omega; \mathbb{R}^d), \boldsymbol{\tau}=\boldsymbol{\tau}^{\intercal}\}. $$
The transition from the primal displacement formulation to this mixed system, including the definition of the compliance tensor $\mathcal{A}$, is detailed in Section 2.

Discretizing this mixed formulation is difficult because the scheme must: (i) satisfy the inf-sup condition; (ii) ensure stress symmetry and normal continuity; and (iii) provide optimal order of convergence. Many finite elements have been developed to meet these requirements \cite{ArnoldAwanouWintherMC2008,ArnoldWintherNM2002,ChenHuang2022,ChenHuang2024,HuZhangSCM2015,HuangZhangZhouZhuACM2024}. 
However, these elements typically include vertex degrees of freedom (DoFs), which makes them non-hybridizable.

To avoid vertex DoFs, hybridizable $H(\div)$-conforming symmetric stress elements on barycentric refinements were introduced in~\cite{ArnoldDouglasGupta1984,ChenHuang2025,ChristiansenHu2023} 
or the Worsey-Farin split, which in three dimensions divides each tetrahedron into twelve sub-tetrahedra~\cite{GongGopalakrishnanGuzmanNeilan2023}. Other methods use rational shape functions~\cite{GuzmanNeilan2014} or virtual elements for symmetric tensors~\cite{DassiLovadinaVisinoni2020}. A hybridizable method was also presented in~\cite{GongWuXu2019}, but its stability depends on Scott-Vogelius Stokes elements~\cite{ScottVogelius1985,Zhang2005} 
on specific meshes.

For completeness, we also note normal-normal continuous elements~\cite{CarstensenHeuer2025}, TDNNS method~\cite{PechsteinSchoeberl2011,PechsteinSchoeberl2018}, nonconforming mixed finite elements~\cite{ArnoldAwanouWinther2014,ArnoldWintherM32003,GopaGuzmSiam2011,HuMa2019,YiM32006}, 
and discontinuous Galerkin methods~\cite{CarstHellSiam2016,HuangHuangANM2018,SoonCockburnStolIJNAM2009,WangWuXuJSC2020} 
for elasticity.

Managing continuity across mesh boundaries is central to finite element design. Standard mixed methods use strong continuity for one variable and weak continuity for the other. In contrast, the SDG framework \cite{ChungEngquist06,ChungWave09} uses a primal-dual mesh structure where each variable is continuous only on its own mesh boundaries. For mixed elasticity, \cite{LeeKimJSC2016,ZhaoParkSIAM2020} showed that the stress tensor can be discontinuous across primal edges but remains normal-continuous on dual edges, while the displacement follows the opposite pattern. This staggered continuity ensures stability without extra stabilization terms. However, current SDG methods for elasticity only enforce weak symmetry rather than strong symmetry of the stress tensor \cite{LeeKimJSC2016,ZhaoParkSIAM2020}.

In this work, we develop a new SDG scheme for mixed linear elasticity that interchanges the continuity assignments of stress and displacement compared with \cite{LeeKimJSC2016,ZhaoParkSIAM2020}. Let $\mathcal K_h$ be a polyhedral mesh and $\mathcal K_h^*$ be its dual mesh by adding a vertex to the center of each element in $\mathcal K_h$. In the proposed new SDG method, the stress tensor is piecewise $\mathbb P_k(T; \mathbb S)$ but in $H(\div,\mathbb S; \omega_F)$ for the patch containing primal face $F$. The stress is strongly symmetric, in contrast to the weak symmetry employed in previous SDG formulations~\cite{LeeKimJSC2016,ZhaoParkSIAM2020}. The displacement, is approximated by piecewise polynomials over general polytopal elements which is in $\mathbb P_{k+1}(K;\mathbb R^d)\subset H^1(K;\mathbb R^d)$ for the primal cell $K\in \mathcal K_h$. The normal continuity of stress is on the boundary of the primal element, while the displacement is continuous on the boundary of dual elements. This complementary continuity ensures a discrete inf-sup condition. 

In classical linear elasticity, the equilibrium equation $
\operatorname{div}\boldsymbol{\sigma} = \boldsymbol{f}$
expresses the balance of linear momentum. 
The symmetry of the stress tensor, $\boldsymbol{\sigma}=\boldsymbol{\sigma}^{\intercal}$, follows from the balance of angular momentum. 
Together, these balance laws provide the physical basis for symmetric, $H(\mathrm{div})$-conforming stress formulations. In our scheme, both important balance law are preserved locally.

Reversing continuity assignments in an SDG framework was recently applied to the Poisson equation \cite{ChenHuangParkWangJSC2025}. That work highlights links between this SDG design and nonconforming Crouzeix-Raviart, weak Galerkin, and virtual element methods. Extending this approach to linear elasticity is difficult. It requires symmetric-tensor DoFs that are normal-continuous across polytopal boundaries. The method is locking-free; its stability and convergence are robust in the incompressible limit $\lambda\to\infty$.

The remainder of the paper is organized as follows.
Section~\ref{sec:space} introduces the finite element spaces used in the SDG discretization and establishes their unisolvence.
Section~\ref{sec:SDG} presents the SDG scheme and proves its well-posedness.
Section~\ref{sec:hybrid} describes the hybridization procedure and derives optimal a priori error estimates.
Section~\ref{sec:numerexam} reports numerical experiments that validate the theoretical results. 
Section~\ref{sec:conclusion} provides concluding remarks.

Throughout this paper, the notation ``$\lesssim\cdots $" means ``$\leq C\cdots$" for a generic constant $C > 0$ independent of the mesh size $h$ and the Lam\'{e} parameter $\lambda$, though its value may vary in different contexts. The equivalence ``$A\eqsim B$" denotes both ``$A\lesssim B$" and ``$B\lesssim A$".

\section{Finite Element Spaces}\label{sec:space}
This section introduces the approximation spaces for the displacement and stress variables with focus on the stress space and its associated DoFs, which preserve symmetry and normal continuity across polytopal elements. 

\subsection{Linear elastic equations}
The linear elasticity problem is given by
\begin{equation}\label{elasproblem}
-\div(2\mu\boldsymbol{\varepsilon}(\boldsymbol{u})+\lambda(\div\boldsymbol{u})\boldsymbol{I})=\boldsymbol{f} \;\;\textrm{ in } \Omega;\quad \boldsymbol{u}|_{\partial\Omega}=\boldsymbol{0},
\end{equation}
where $\Omega\subset \mathbb{R}^d\, (d\geq 2)$ is a bounded polytope. Here $\boldsymbol{u}$ denotes the displacement field, $\boldsymbol{f}$ is the body force, $\mu$ and $\lambda$ are the Lam\'{e} constants,  and $\boldsymbol{I}$ is the identity matrix.
Introducing the stress tensor 
$\boldsymbol{\sigma}=2\mu\boldsymbol{\varepsilon}(\boldsymbol{u})
+\lambda(\div\boldsymbol{u})\boldsymbol{I}$,
the problem \eqref{elasproblem} can be rewritten as the first-order system:
\begin{equation}
\left\{
\begin{aligned}
-\div\boldsymbol{\sigma}&=\boldsymbol{f} \; \qquad\text{in} \;\Omega ,
\\
\mathcal{A}\boldsymbol{\sigma}&=\boldsymbol{\varepsilon}(\boldsymbol{u}) \; \,\,\;\,\text{in} \;\Omega,
\\
\boldsymbol{u}&=\boldsymbol{0} \,\;\;\;\;\, \quad\text{on} \;\partial\Omega,\label{equ_elasticity}
\end{aligned}
\right.
\end{equation}
where
$$
\mathcal{A}\boldsymbol{\sigma}=\frac{1}{2\mu}\boldsymbol{\sigma}-\frac{\lambda}{2\mu(2\mu+d\lambda)}\tr(\boldsymbol{\sigma})\boldsymbol{I} = \frac{1}{2\mu}\dev\boldsymbol{\sigma}+\frac{1}{d(2\mu+d\lambda)}\tr(\boldsymbol{\sigma})\boldsymbol{I},
$$ 
$\tr(\boldsymbol{\sigma})$ denotes the trace of $\boldsymbol{\sigma}$, and
\begin{equation}\label{eqn_sigma_h_D}
\dev\boldsymbol{\sigma}:=\boldsymbol{\sigma} - \frac{1}{d}\tr(\boldsymbol{\sigma})\boldsymbol{I}.
\end{equation}
The mixed variational formulation \eqref{equ_mixed_formulation} is derived from the Hellinger-Reissner mixed variational principle.

\subsection{Primal and dual meshes}
Let $\mathcal{K}_h$ be a shape-regular polytopal mesh of the domain $\Omega$; see Fig.~\ref{partitions1}. We assume each element $K\in\mathcal{K}_h$ is a star-shaped polytope with a bounded chunkiness parameter. Let $\mathcal{F}_h$ be the set of $(d-1)$-dimensional faces of $\mathcal{K}_h$, and $\mathring{\mathcal{F}}_h^{K}=\mathcal{F}_h \setminus \partial \Omega$. We call $\mathcal{K}_h$ the primal mesh.

For each $K\in\mathcal{K}_h$, let $\boldsymbol{x}_K$ be the center of the largest ball inside $K$. Connecting $\boldsymbol{x}_K$ to the vertices of $K$ gives a shape-regular refinement $\mathcal{T}_h$; see Fig.~\ref{partitions2}. Let $\mathcal{T}_h(K)$ be the set of pyramids in $K$. Let $\mathcal{F}_h^T$ be the set of $(d-1)$-dimensional faces of $\mathcal{T}_h$, and $\mathring{\mathcal{F}}_h^{T}=\mathcal{F}_h^T \setminus \partial \Omega$. Let $\lambda_{\boldsymbol{x}_K}$ be the piecewise linear function on $\mathcal T_h(K)$ such that $\lambda_{\boldsymbol{x}_K}(\boldsymbol{x}_K)=1$ and $\lambda_{\boldsymbol{x}_K}|_F=0$ for any primal face $F\in \mathcal{F}_h$. 

Faces in $\mathcal F_h^K$ are primal faces. For each $F\in\mathcal{F}_h$, let $\omega_F$ be the union of pyramids in $\mathcal{T}_h$ sharing $F$. These diamond-shaped regions form the dual mesh:
\begin{equation*}
\mathcal{K}_h^*=\{\omega_F, F\in \mathcal{F}_h\}.
\end{equation*}
Faces of $\partial \omega_F$ are dual faces and the set is denoted by $\mathcal F_h^*$; see Fig.~\ref{partitions3}. By definition, $\mathcal F_h^*\cap \mathcal F_h$ consists of faces on the boundary. 

\begin{figure}[hbtp!]
\label{partitions}
\subfigure[The primal mesh $\mathcal{K}_h$.]{
\begin{minipage}[t]{0.315\linewidth}
\centering
\includegraphics*[width=3.75cm]{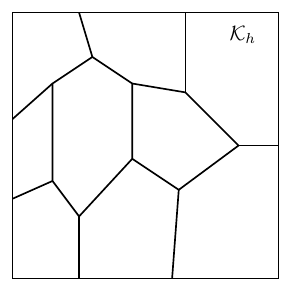} \label{partitions1}
\quad
\end{minipage}}
\subfigure[The triangulation mesh $\mathcal{T}_h$.]
{\begin{minipage}[t]{0.315\linewidth}
\centering
\includegraphics*[width=3.75cm]{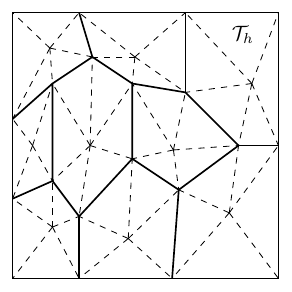} \label{partitions2}
\end{minipage}}
\subfigure[The dual mesh $\mathcal{K}_h^*$.]
{\begin{minipage}[t]{0.315\linewidth}
\centering
\includegraphics*[width=3.75cm]{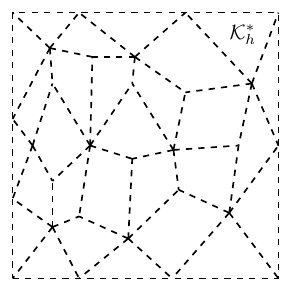} \label{partitions3}
\end{minipage}}
\caption{The polygonal and triangle partitions.}
\end{figure}

\subsection{Tangential-normal decomposition}
For $K\in\mathcal{K}_h$, let $\mathbb P_k(K)$ be the space of polynomials on $K$ of degree at most $k$. Let $\mathbb{P}_k(K; \mathbb{X})=\mathbb{P}_k(K)\otimes \mathbb{X}$, where $\mathbb{X}$ is $\mathbb{R}^d$ for vectors, $\mathbb{S}$ for symmetric matrices, and $\mathbb{K}$ for skew-symmetric matrices. 

Let $F \in \partial K$ be a face with unit normal $\boldsymbol{n}_F$ and orthonormal tangent vectors $\{\boldsymbol{t}_i^{F}\}_{i=1}^{d-1}$. For any $\boldsymbol{v}\in\mathbb{R}^d$, the tangential part on $F$ is
$$
\Pi_F \boldsymbol{v}=(\boldsymbol{I}-\boldsymbol{n}_F\otimes\boldsymbol{n}_F)\boldsymbol{v}
=\sum_{i=1}^{d-1}(\boldsymbol{v}\cdot\boldsymbol{t}_i^{F})\boldsymbol{t}_i^{F}.
$$

Define the following spaces on $F$:
\begin{align*}
\mathscr T^F&:=\textrm{span}\{\boldsymbol{t}_1^F, \ldots, \boldsymbol{t}_{d-1}^F\}, \\
\mathscr T^F(\mathbb S)&:=\textrm{span}\{{\rm sym}(\boldsymbol{t}_i^F\otimes \boldsymbol{t}_j^F) \;\;\textrm{ for } 1\leq i\leq j\leq d-1\}, \\
\mathscr N^F(\mathbb S) &:= \sym(\mathscr T^F\otimes\boldsymbol{n}_F)\oplus {\rm span}\{\boldsymbol{n}_F\otimes\boldsymbol{n}_F\},\\
\mathscr T^F(\mathbb K)&:=\textrm{span}\{{\rm skw}(\boldsymbol{t}_i^F\otimes \boldsymbol{t}_j^F) \;\;\textrm{ for } 1\leq i< j\leq d-1\}.
\end{align*}
These subspaces give the following $\boldsymbol{t}$-$\boldsymbol{n}$ decompositions
\begin{align}\label{eq:decS}
\mathbb S&= \mathscr T^F(\mathbb S)\oplus \mathscr N^F(\mathbb S), \quad 
\\
\mathbb K&= \mathscr T^F(\mathbb K)\oplus \mathscr N^F(\mathbb K), \quad \mathscr N^F(\mathbb K)={\rm skw}(\mathscr T^F\otimes\boldsymbol{n}_F).
\end{align}

\subsection{Local finite element spaces for symmetric stress}
In this subsection, we assume $k\geq 1$. We define the DoFs \eqref{SigmaKDoFs} for the local stress space $\Sigma_h^{-1}(K)$. 

Let $T\in\mathcal{T}_h(K)$ be a sub-element of the polytope $K\in\mathcal{K}_h$, and let $F=\partial K \cap \partial T$ be their common boundary face. We define $\lambda_F$ as the level set function for the face $F$ with respect to $K$; that is, 
$\lambda_{F}$ is a linear function on $T$ satisfying $\lambda_{F}(\boldsymbol{x}_K)=1$ and $\lambda_{F}|_F=0$. We have the decomposition $\mathbb{P}_{k}(T)=\mathbb{P}_{k}(F)\oplus\lambda_F\mathbb{P}_{k-1}(T)$.

Let $K \in \mathcal{K}_h$ be a polytope with $n$ boundary faces $F_i$ for $i=1,\ldots,n$. For each face $F_i$, let $T_i \in \mathcal{T}_h(K)$ be the sub-element such that $F_i \subset \partial T_i$. Let $\boldsymbol{n}_i$ and $\boldsymbol{t}_i^j$ ($j=1,\ldots,d-1$) be the unit normal and tangent vectors of $T_i$ on $F_i$. We define the local symmetric tensor space on $K$ as the discontinuous polynomial space
\begin{align}\label{eq:202411132}
\Sigma_h^{-1}(K)&=\Oplus_{i=1}^n\chi_{T_i}\Sigma_k(T_i)=\{\boldsymbol{\tau}\in L^2(K;\mathbb S): \boldsymbol{\tau}|_T\in \Sigma_k(T) \textrm{ for }T\in\mathcal{T}_h(K)\},
\end{align}
where $\chi_{T_i}$ is the characteristic function of $T_i$ and $\Sigma_k(T)=\mathbb P_k(T;\mathbb S)$. We  define the bubble space
$$ \mathbb{B}_k(\div,K;\mathbb{S})=\{\boldsymbol{\tau}\in \Sigma_h^{-1}(K): \tr\boldsymbol{\tau} :=\boldsymbol{\tau}\boldsymbol{n}|_{\partial K}=\boldsymbol{0}\}.$$

Consider the trace on $F$: $\tr_F\boldsymbol{\tau}=(\boldsymbol{\tau}\boldsymbol{n}_F)|_F$. We have $\lambda_F|_F = 0$, $\tr_F \mathscr T^F(\mathbb S) = 0$, and $\tr_F \mathbb{P}_{k}(F)\otimes \mathscr N^F(\mathbb S) = \mathbb P_k(F; \mathbb R^d) =: V_k(F)$. Therefore, we have the following characterization
\begin{equation}\label{eq:bubblekKgeodec}
\begin{aligned}
\mathbb{B}_k(\div,K;\mathbb{S})={}&\Oplus_{i=1}^n\lambda_{\boldsymbol{x}_K}\chi_{T_i}\sym(\mathbb{P}_{k-1}(T_i;\mathbb R^d)\otimes\boldsymbol{n}_i) \\
\oplus &\Oplus_{i=1}^n\chi_{T_i}\mathbb{P}_{k}(T_i;\mathscr T^{F_i}(\mathbb S)).
\end{aligned}
\end{equation}

\begin{lemma}\label{lemma_Geo_Dec}
For $K\in\mathcal{K}_h$, we have the geometric decomposition
\begin{align}
\label{eq:SigmakKgeodec}
\Sigma_h^{-1}(K)&=\mathbb{B}_k(\div,K;\mathbb{S})\oplus\Oplus_{i=1}^n\chi_{T_i}\sym( V_{k}(F_i)\otimes\boldsymbol{n}_i), 
\end{align}
and the trace characterization
\begin{equation}\label{eq:SigmaktracepartK}
  \tr(\Sigma_h^{-1}(K))=\prod_{i=1}^n V_{k}(F_i).
\end{equation}
\end{lemma}

Using \eqref{eq:SigmakKgeodec} and \eqref{eq:SigmaktracepartK}, $\Sigma_h^{-1}(K)$ is uniquely determined by the DoFs:
\begin{subequations}
\label{equ_dofs_dd_polytope}
\begin{align}
\label{equ_dofs_dd_polytope1}
\int_{F}\boldsymbol{\sigma}\boldsymbol{n}\cdot \boldsymbol{q}\ds,&\quad \boldsymbol{q}\in   V_{k}(F), F\in \partial K, \\
\label{equ_dofs_dd_polytope2}
\int_K\boldsymbol{\sigma}:\boldsymbol{\tau} \dx, & \quad  \boldsymbol{\tau}\in \mathbb{B}_k(\div ,K;\mathbb{S}).
\end{align}
\end{subequations}

To prove the discrete inf-sup condition, we shall separate the DoFs for $\mathbb{P}_k(K;\mathbb{S})$ from the component $\Oplus_{i=1}^n\chi_{T_i}\mathbb{P}_{k}(T_i;\mathscr T^{F_i}(\mathbb S))$ for the bubble space $\mathbb{B}_k(\div ,K;\mathbb{S})$.
The following construction is only theoretical. With hybridization, we do not need to figure out the interior DoFs exactly. 

\begin{assumption}\label{meshassumption}
There exist $d(d+1)/2$ symmetric tensors $b_\ell$ such that
\begin{equation}
\mathbb{S}=\text{span}\left\{ b_{\ell}, \ell=1,\ldots,d(d+1)/2\right\},\label{equ_assumption_ddpolytope}
\end{equation}
where $\mathscr T^{F_i}(\mathbb S)=\textrm{span}\{ b_i^{m}, m=1,\ldots,d(d-1)/2\}$ for $F_i\in \partial K$, and 
\begin{equation}
\left\{ b_\ell\right\} \subset \left\{ b_i^m, i=1,\ldots,n, m=1,\ldots,d(d-1)/2\right\}.\label{equ_edge_vector_property}
\end{equation}
\end{assumption}

\begin{lemma}
Let $K$ be a 2D polygon. Then $K$ satisfies Assumption \ref{meshassumption} if $K$ have at least three edges that are pairwise non-collinear.
\end{lemma}
\begin{proof}
Take three edges of K and denote them as $\boldsymbol{e}_1, \boldsymbol{e}_2, \boldsymbol{e}_3$. Since $\boldsymbol{e}_1, \boldsymbol{e}_2, \boldsymbol{e}_3$ are pairwise non-collinear, their extensions must intersect to form a triangle. Then, according to \cite[Lemma 2.1]{Hu_JCM_2015}, we have 
$$
\mathbb{S}=\textrm{span}\{\boldsymbol{e}_1\otimes \boldsymbol{e}_1,\boldsymbol{e}_2\otimes \boldsymbol{e}_2,\boldsymbol{e}_3\otimes \boldsymbol{e}_3\},
$$ 
which finishes the proof.
\end{proof}

This assumption is necessary because a 2D rectangle or a 2D parallelogram, having only two independent tangential vectors, cannot span the space of 2D symmetric matrices. Symmetric stress elements for 2D rectangles are constructed in \cite{HuSIAM20152}.

\begin{lemma}
Let $K$ be a 3D polyhedron. Then $K$ satisfies Assumption \ref{meshassumption}.
\end{lemma}
\begin{proof}
For any polyhedron that is homeomorphic to a sphere, equivalently with Euler characteristic $2$, the sum of the angle defects over all vertices equals $4\pi$, as given by Descartes' theorem or, more generally, the Gauss--Bonnet theorem. Since this total defect is positive, there must exist at least one vertex with a positive angle defect. By definition of a polyhedron, each vertex is the intersection of at least three faces, and each face contributes an incident edge, so at least three edges also meet at that vertex. Therefore, any spherical polyhedron necessarily has a vertex where at least three faces and three edges meet and where the angle defect is positive. 

Let $\texttt{v}_0$ be such a vertex of $K$ where $n \geq 3$ faces meet. Let $\{\boldsymbol{e}_1, \boldsymbol{e}_2, \dots, \boldsymbol{e}_n\}$ be the set of edges incident to $\texttt{v}_0$, ordered cyclically. By the topology of a polyhedron, each pair of adjacent edges $\{\boldsymbol{e}_i, \boldsymbol{e}_{i+1}\}$ (with $\boldsymbol{e}_{n+1}=\boldsymbol{e}_1$) defines a boundary face $F_i$.

If $n=3$, the six vectors in $\mathcal{E}_{local}:=\{\boldsymbol{e}_1, \boldsymbol{e}_2, \boldsymbol{e}_3, \boldsymbol{e}_2-\boldsymbol{e}_3, \boldsymbol{e}_3-\boldsymbol{e}_1, \boldsymbol{e}_1-\boldsymbol{e}_2\}$ form a non-degenerate tetrahedron, which spans $\mathbb{S}$ \cite{Hu_JCM_2015}.

For $n>3$, let
\begin{equation*}
\hat{\mathbb S}=\textrm{span}\{\boldsymbol{e}_1\otimes \boldsymbol{e}_1,\boldsymbol{e}_2\otimes \boldsymbol{e}_2,\boldsymbol{e}_3\otimes \boldsymbol{e}_3,\sym(\boldsymbol{e}_1\otimes \boldsymbol{e}_2),\sym(\boldsymbol{e}_2\otimes \boldsymbol{e}_3) \}\subset \cup_{F\in \partial K} \mathscr T^{F}(\mathbb S).
\end{equation*}

We have one more edge $\boldsymbol{e}_4$:
\begin{equation}\label{eq:e4e1-3}
\boldsymbol{e}_4=c_1\boldsymbol{e}_1+c_2\boldsymbol{e}_2+c_3\boldsymbol{e}_3\quad \textrm{with } c_1\neq0, c_3\neq0; 
\end{equation}
otherwise, it will reduce to the case $n=3$.

Next we prove that 
\begin{equation}\label{eq:S3d}
\mathbb S=\hat{\mathbb S}\oplus\textrm{span}\{\boldsymbol{e}_4\otimes \boldsymbol{e}_4\}.
\end{equation}
By \eqref{eq:e4e1-3}, $\boldsymbol{e}_4\otimes \boldsymbol{e}_4 - 2c_1c_3\sym(\boldsymbol{e}_1\otimes \boldsymbol{e}_3)\in\hat{\mathbb S}$. This together with the fact $\mathbb S=\hat{\mathbb S}\oplus\textrm{span}\{\sym(\boldsymbol{e}_1\otimes \boldsymbol{e}_3)\}$
yields \eqref{eq:S3d}.
\end{proof}

We now formally define the DoFs for the local stress space $\Sigma_h^{-1}(K)$.

\begin{lemma}\label{lemma_DOF}
For $K\in\mathcal{K}_h$ and $k\geq1$, the space $\Sigma_h^{-1}(K)$ is uniquely determined by the following DoFs:
\begin{subequations}
\label{SigmaKDoFs}
\begin{align}
\label{SigmaKDoFs1}
\int_{F}\boldsymbol{\sigma}\boldsymbol{n}\cdot \boldsymbol{q}\ds,&\quad \boldsymbol{q}\in   V_{k}(F), F\in \partial K , \\
\label{SigmaKDoFs2}
\int_K\boldsymbol{\sigma}:\boldsymbol{\tau} \dx,& \quad  \boldsymbol{\tau}\in \mathbb{B}_k(\div,K;\mathbb{S})\setminus\Oplus_{\ell=1}^{d(d+1)/2}\left (\chi_{T_{\boldsymbol{e}_\ell}} \mathbb{P}_{k}(T_{\boldsymbol{e}_\ell})b_\ell \right ), \\
\label{SigmaKDoFs3}
\int_K\boldsymbol{\sigma}:\boldsymbol{\tau} \dx,& \quad  \boldsymbol{\tau}\in \mathbb{P}_k(K;\mathbb{S}),
\end{align}
\end{subequations}
where $T_{\boldsymbol{e}_\ell}\in\mathcal{T}_h(K)$ is a sub-element containing the face basis $b_\ell$ chosen in \eqref{equ_edge_vector_property}.
\end{lemma}

\begin{proof}
Comparing \eqref{SigmaKDoFs} with \eqref{equ_dofs_dd_polytope}, it suffices to show that the bubble space $\mathbb{B}_k(\div, K; \mathbb{S})$ is uniquely determined by the DoFs \eqref{SigmaKDoFs2}--\eqref{SigmaKDoFs3}. 

The number of DoFs in \eqref{SigmaKDoFs2}--\eqref{SigmaKDoFs3} equals $\dim\mathbb{B}_k(\div ,K;\mathbb{S})$. Let us assume $\boldsymbol{\sigma}\in\mathbb{B}_k(\div ,K;\mathbb{S})$ with all DoFs in \eqref{SigmaKDoFs2}--\eqref{SigmaKDoFs3} vanish. The vanishing of \eqref{SigmaKDoFs2} implies 
$$ \boldsymbol{\sigma}=\sum_{\ell=1}^{d(d+1)/2}\chi_{T_{\boldsymbol{e}_\ell}} q_\ell b_\ell \in\Oplus_{\ell=1}^{d(d+1)/2}\left (\chi_{T_{\boldsymbol{e}_\ell}} \mathbb{P}_{k}(T_{\boldsymbol{e}_\ell})b_\ell \right ),$$
where $q_\ell\in\mathbb{P}_{k}(T_{\boldsymbol{e}_\ell})$. We extend the domain of $q_\ell$ such that $q_\ell\in\mathbb{P}_{k}(K)$. 

Let $\{\tilde{\tau}_\ell\}_{\ell=1}^{d(d+1)/2}$ be the dual basis to $\{b_\ell\}_{\ell=1}^{d(d+1)/2}$. By testing with $\boldsymbol{\tau}=q_\ell\tilde{\tau}_\ell\in \mathbb{P}_k(K;\mathbb{S})$ in \eqref{SigmaKDoFs3}, we obtain
$$ \|q_\ell\|_{T_{\boldsymbol{e}_\ell}}^2=0, \quad \ell=1,\ldots,d(d+1)/2. $$
Thus, $q_\ell=0$ for all $\ell$, which implies $\boldsymbol{\sigma}=\boldsymbol{0}$.
\end{proof}

\subsection{Global finite element spaces}
For $F\in \mathring{\mathcal{F}}_h^{K} \subset \mathring{\mathcal{F}}_h^{T}$, let $T_1$ and $T_2$ be the two elements in $\omega_F$ such that $F=\partial T_1 \cap \partial T_2$. Let $\boldsymbol{u}_h^i = \boldsymbol{u}_h|_{T_i}$ for $i=1,2$. For any $\boldsymbol{u}_h\in L^2(\Omega;\mathbb{R}^d)$, the jump operator across $F$ is defined as
$$
[\boldsymbol{u}_h]_F=\left\{
\begin{array}{ll}
\boldsymbol{u}_h^1-\boldsymbol{u}_h^2,\; & F\in \mathring{\mathcal{F}}_h^{K}, \\
\boldsymbol{u}_h,\; & F\in \partial\Omega.
\end{array}
\right.
$$

Similarly, for any $F \in \mathcal{F}_h^{T}$, let $T_1, T_2 \in \mathcal{T}_h$ be the two elements sharing face $F$, and let $\boldsymbol{n}_i$ be the outward unit normal of $T_i$ on $F$. For a piecewise smooth tensor field $\boldsymbol{\sigma}_h$, the normal jump is
$$
[\boldsymbol{\sigma}_h \boldsymbol{n}]_F=\left\{
\begin{array}{ll}
\boldsymbol{\sigma}_h^1 \boldsymbol{n}_1+\boldsymbol{\sigma}_h^2 \boldsymbol{n}_2, & F\in \mathring{\mathcal{F}}_h^{T}, \\
\boldsymbol{\sigma}_h \boldsymbol{n}_{\partial \Omega},& F\in \partial\Omega,
\end{array}
\right.
$$
where $\boldsymbol{\sigma}_h^{i} = \boldsymbol{\sigma}_h|_{T_i}$.

The approximation space for the displacement $\boldsymbol{u}$ consists of vector-valued piecewise polynomials over the primal mesh $\mathcal{K}_h$:
\begin{equation*}
U_h=\{\boldsymbol{u}_h\in L^2(\Omega;\mathbb{R}^d):\; \boldsymbol{u}_h|_K\in \mathbb{P}_{k+1}(K;\mathbb{R}^d), K\in \mathcal{K}_h\}=\prod_{K\in\mathcal{K}_h}\mathbb{P}_{k+1}(K;\mathbb{R}^d).
\end{equation*}
Functions in $U_h$ are polynomials within each polytope $K$ but may be discontinuous across $\partial K$. Let $Q_{k+1}$ be the $L^2$ projection onto $U_h$. To enforce stability, functions in the stress space $\Sigma_h$ are required to have continuous normal flux across the primal faces. This global stress space is defined as
\begin{equation*}
\Sigma_h = \{\boldsymbol{\sigma}_h\in L^2(\Omega;\mathbb{S}):\; \boldsymbol{\sigma}_h|_T\in \Sigma_k(T),\; T\in \mathcal{T}_h,\,
[\boldsymbol{\sigma}_h\boldsymbol{n}]|_F=0, \;F\in \mathring{\mathcal{F}}_h^{K}\}.
\end{equation*}

\subsection{Finite element spaces of the lowest order}
In the remainder of this section, we extend the approximation spaces to the lowest-order ($k=0$) case. To ensure the well-posedness of the corresponding numerical scheme, we incorporate the space of rigid motions.

The kernel of the strain operator $\boldsymbol{\varepsilon}(\cdot)$ on an element $K$ is the rigid motion space~\cite{Nedelec1980}:
$$
{\rm RM}(K)=\mathbb{P}_0(K; \mathbb{R}^d) \oplus \mathbb{P}_0(K; \mathbb{K})\boldsymbol{x},
$$
where $\boldsymbol{x}$ represents the position vector. For any face $F\in\partial K$, the polynomial space $\mathbb P_k(F)$ can be constructed using the following basis:
\begin{equation*}
\mathbb P_k(F)=\textrm{span}\big\{(\boldsymbol{t}_1^F\cdot\boldsymbol{x})^{\alpha_1}\cdots(\boldsymbol{t}_{d-1}^F\cdot\boldsymbol{x})^{\alpha_{d-1}}: \sum_{i=1}^{d-1}\alpha_i \leq k, \alpha_i\in\mathbb N\big\}.
\end{equation*}
Based on this basis, we naturally extend the domain of polynomials in $\mathbb P_k(F)$ from $F$ to $K$ (or to a sub-element $T\in\mathcal T_h(K)$ sharing $F$) by assuming the function is constant along lines perpendicular to $F$.

The rigid motion space on face $F$ is defined as:
\begin{equation*}
{\rm RM}(F):=\mathbb{P}_0(F; \mathscr T^F) \oplus \mathbb{P}_0(F; \mathscr T^F(\mathbb K)) \boldsymbol{x}.
\end{equation*}
It follows that ${\rm RM}(F)=\Pi_F {\rm RM}(K)$ \cite[Lemma 11]{HuangZhangZhouZhuACM2024}, which allows us to interpret ${\rm RM}(F)$ as being defined on either the face $F$ or the polytope $K$.

\begin{lemma}
For any face $F\in\partial K$, we have the following decomposition:
\begin{equation}\label{eq:RMdecomp}
{\rm RM}(K)|_F=\mathbb{P}_1(F)\boldsymbol{n}_F\oplus {\rm RM}(F).
\end{equation}
\end{lemma}

\begin{proof}
Given that ${\rm RM}(F)=\Pi_F {\rm RM}(K)$, the inclusion ${\rm RM}(K)|_F\subseteq\mathbb{P}_1(F)\boldsymbol{n}_F\oplus {\rm RM}(F)$ is immediate. To prove the reverse inclusion, first note that
$$
\mathbb{P}_0(F)\boldsymbol{n}_F \oplus {\rm RM}(F) \subseteq {\rm RM}(K)|_F.
$$
For $i=1,\dots,d-1$, since $\boldsymbol{n}_F \cdot \boldsymbol{x}$ is constant on $F$, we have
$$
(\boldsymbol{t}_i^F \cdot \boldsymbol{x}) \boldsymbol{n}_F = \big((\boldsymbol{n}_F \otimes \boldsymbol{t}_i^F - \boldsymbol{t}_i^F \otimes \boldsymbol{n}_F)\boldsymbol{x} + (\boldsymbol{n}_F \cdot \boldsymbol{x}) \boldsymbol{t}_i^F\big)\big|_F \in {\rm RM}(K)|_F.
$$
Hence, the inclusion $\mathbb{P}_1(F)\boldsymbol{n}_F \oplus {\rm RM}(F) \subseteq {\rm RM}(K)|_F$ holds, which completes the proof.
\end{proof}

For $T\in\mathcal{T}_h$, the lowest-order local stress space is defined as
\begin{align*}
\Sigma_0(T)
&= \sym(V_{0}(F)\otimes\boldsymbol{n}_F) \oplus \mathbb{P}_{0}(T;\mathscr T^F(\mathbb S)),
\end{align*}
where $F\in\mathcal{F}_h^{K}\cap \partial T$ and
$$V_{0}(F) := {\rm RM}(K)|_F=\mathbb{P}_1(F) \boldsymbol{n}_F \oplus {\rm RM}(F).$$
We have 
\begin{equation*}
\dim \Sigma_0(T)=d^2,\quad \dim V_{0}(F)=\frac{1}{2}d(d+1).
\end{equation*}
The trace of $\Sigma_0(T)$ onto the face $F\in\mathcal{F}_h^{K}\cap \partial T$ is then given by
\begin{equation*}
\tr_F(\Sigma_0(T))=V_{0}(F). 
\end{equation*}
The space $\Sigma_0(T)$ is an enrichment of $ \mathbb{P}_0(T; \mathbb S)$, designed so that the trace operator defines a semi-norm. Together with the strain operator, this results in a norm that satisfies the piecewise Korn's inequality \eqref{eq:KornineqltypieceH1}.
The enrichment $\Sigma_0(T)$ is divergence free.
This property is inherited from the ${\rm RM}(K)$ space, due to its construction via an extension of the trace ${\rm RM}(K)|_F$.

\begin{lemma}
For $T\in\mathcal{T}_h$, we have
\begin{equation}\label{eq:divSigma0}
\div\Sigma_0(T)=\{0\}.
\end{equation}
\end{lemma}
\begin{proof}
It suffices to demonstrate that
\begin{align*}
\div\big(\boldsymbol{n}_F\otimes\boldsymbol{n}_F(\boldsymbol{t}_i^F\cdot\boldsymbol{x})\big)&=0,\quad i=1,\ldots, d-1, \\
\div\sym\big(\boldsymbol{n}_F\otimes (\boldsymbol{\tau}\boldsymbol{x})\big)&=0, \quad \boldsymbol{\tau}\in\mathscr T^F(\mathbb K).
\end{align*}
Using the identity $(\boldsymbol{n}_F\cdot\nabla)\boldsymbol{x}=\boldsymbol{n}_F$, we obtain
\begin{equation*}
\div\big(\boldsymbol{n}_F\otimes\boldsymbol{n}_F(\boldsymbol{t}_i^F\cdot\boldsymbol{x})\big) = \boldsymbol{n}_F(\boldsymbol{n}_F\cdot\nabla)(\boldsymbol{t}_i^F\cdot\boldsymbol{x}) = \boldsymbol{n}_F(\boldsymbol{t}_i^F\cdot\boldsymbol{n}_F)=0.
\end{equation*}
Furthermore, since $\boldsymbol{\tau}\boldsymbol{x}\in {\rm RM}(T)\subseteq \mathbb{P}_1(T; \mathbb{R}^d)\cap\ker(\div)$, it follows that
\begin{equation*}
\div\sym\big(\boldsymbol{n}_F\otimes (\boldsymbol{\tau}\boldsymbol{x})\big) = \frac{1}{2}(\boldsymbol{n}_F\cdot\nabla)(\boldsymbol{\tau}\boldsymbol{x}) = \frac{1}{2}\boldsymbol{\tau}\boldsymbol{n}_F=0.
\end{equation*}
This completes the proof.
\end{proof}

For $K\in\mathcal{K}_h$, we define:
\begin{align*}
\Sigma_h^{-1}(K)&=\Oplus_{i=1}^n\chi_{T_i}\Sigma_0(T_i),\\
\mathbb{B}_0(\div,K;\mathbb{S})&=\{\boldsymbol{\tau}\in \Sigma_h^{-1}(K): \boldsymbol{\tau}\boldsymbol{n}|_{\partial K}=\boldsymbol{0}\}.
\end{align*}
Following the logic of Lemmas \ref{lemma_Geo_Dec} and \ref{lemma_DOF}, we state the following results.

\begin{lemma}
For $K\in\mathcal{K}_h$ and $k=0$, the geometric decompositions of $\Sigma_h^{-1}(K)$ and $\mathbb{B}_0(\div,K;\mathbb{S})$ are
\begin{align*}
\Sigma_h^{-1}(K)&=\mathbb{B}_0(\div,K;\mathbb{S})\oplus\Oplus_{i=1}^n\chi_{T_i}\sym(V_{0}(F_i)\otimes\boldsymbol{n}_i), \\
\mathbb{B}_0(\div,K;\mathbb{S})&=\Oplus_{i=1}^n\chi_{T_i}\mathbb{P}_{0}(T_i;\mathscr T^{F_i}(\mathbb S)).
\end{align*}
The corresponding trace space is
\begin{equation*}
\tr(\Sigma_h^{-1}(K))=\prod_{i=1}^nV_{0}(F_i).
\end{equation*}
\end{lemma}

\begin{lemma}
For $K\in\mathcal{K}_h$ satisfying Assumption \ref{meshassumption}. The space $\Sigma_h^{-1}(K)$ for $k=0$ is uniquely determined by the following DoFs:
\begin{subequations}
\label{SigmaKDoFsk0}
\begin{align}
\label{SigmaKDoFs1k0}
\int_{F}\boldsymbol{\sigma}\boldsymbol{n}\cdot \boldsymbol{q}\ds,&\quad \boldsymbol{q}\in V_{0}(F), F\in \partial K, \\
\label{SigmaKDoFs2k0}
\int_K\boldsymbol{\sigma}:\boldsymbol{\tau} \dx,& \quad \boldsymbol{\tau}\in \mathbb{B}_0(\div,K;\mathbb{S})\setminus\Oplus_{\ell=1}^{d(d+1)/2}\left (\chi_{T_{\boldsymbol{e}_\ell}}\mathbb{P}_{0}(T_{\boldsymbol{e}_\ell})b_\ell \right ), \\
\label{SigmaKDoFs3k0}
\int_K\boldsymbol{\sigma}:\boldsymbol{\tau} \dx,& \quad \boldsymbol{\tau}\in \mathbb{P}_0(K;\mathbb{S}),
\end{align}
\end{subequations}
where $T_{\boldsymbol{e}_\ell}\in\mathcal{T}_h(K)$ is a sub-element containing the edge vector $\boldsymbol{e}_\ell$ chosen in \eqref{equ_edge_vector_property}.
\end{lemma}

Based on the decomposition \eqref{eq:RMdecomp}, the DoFs in \eqref{SigmaKDoFs1k0} are equivalent to
\begin{equation*}
\int_{F}\boldsymbol{n}^{\intercal} \boldsymbol{\sigma}\boldsymbol{n}\, q \ds,\quad q\in \mathbb{P}_{1}(F);
\qquad
\int_{F} \Pi_F\boldsymbol{\sigma}\boldsymbol{n} \cdot \boldsymbol{q}\ds,\quad \boldsymbol{q}\in {\rm RM}(F).
\end{equation*}

We extend the global approximation spaces $U_h$ and $\Sigma_h$ to $k=0$:
\begin{align*}
U_h &= \{\boldsymbol{u}_h\in L^2(\Omega;\mathbb{R}^d):\; \boldsymbol{u}_h|_K\in \mathbb{P}_{1}(K;\mathbb{R}^d), K\in \mathcal{K}_h\}=\prod_{K\in\mathcal{K}_h}\mathbb{P}_{1}(K;\mathbb{R}^d), \\
\Sigma_h &= \{\boldsymbol{\sigma}_h\in L^2(\Omega;\mathbb{S}):\; \boldsymbol{\sigma}_h|_T\in \Sigma_0(T),\; T\in \mathcal{T}_h,\,
[\boldsymbol{\sigma}_h\boldsymbol{n}]|_F=0, \;F\in \mathring{\mathcal{F}}_h^{K}\}.
\end{align*}

Fig.~\ref{Fig_p0d2qua} illustrates the DoFs for the quadrilateral meshes for $k = 0$.

\begin{figure}[htbp!]
\label{Fig_p0d2qua}
\subfigure[DoFs for $\boldsymbol{\sigma}$ on a quadrangle with $k=0$.]{
\begin{minipage}[t]{0.45\linewidth}
\centering
\includegraphics*[width=4cm]{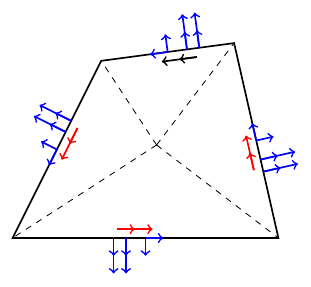}
\end{minipage}}
\subfigure[DoFs for $\boldsymbol{u}$ on a quadrangle with $k=0$.]
{\begin{minipage}[t]{0.45\linewidth}
\centering
\includegraphics*[width=4cm]{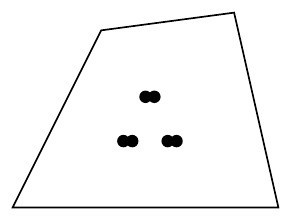}
\end{minipage}}
\caption{ Element diagram for the lowest order $k=0$ stress and displacement on a quadrangle $K$. }
\end{figure}

\subsection{Norm equivalence}
For $k \geq 0$ and face $F\in\mathcal{F}_h$, let $\Pi_k^F$ be the $L^2$-projection operator onto the space $V_k(F)$, defined as:
$$
\Pi_k^F:
\begin{cases}
 L^2(F; \mathbb{R}^d) \rightarrow {\rm RM}(K)|_F, & \textrm{if } k=0, \\
 L^2(F; \mathbb{R}^d) \rightarrow \mathbb{P}_{k}(F; \mathbb{R}^d), & \textrm{if } k\geq 1. 
\end{cases}
$$ 
We introduce the piecewise $H^1$ spaces $H^1(\mathcal{K}_h;\mathbb R^d) := H^1(\mathcal{K}_h)\otimes\mathbb R^d$ and $H^1(\mathcal{T}_h;\mathbb S) := H^1(\mathcal{T}_h)\otimes\mathbb S$, where
\begin{align*}
H^1(\mathcal{K}_h)&:=\{v\in L^2(\Omega): v|_K\in H^1(K) \text{ for each } K\in\mathcal{K}_h\}, \\
H^1(\mathcal{T}_h)&:=\{v\in L^2(\Omega): v|_T\in H^1(T) \text{ for each } T\in\mathcal{T}_h\}.
\end{align*}
The spaces $H^1(\mathcal{K}_h;\mathbb R^d)$ and $H^1(\mathcal{T}_h;\mathbb S)$ are equipped with the following norms, respectively:
\begin{align*}
\|\boldsymbol{v}\|^2_{1,h} &= \sum_{K\in\mathcal{K}_h}\| \boldsymbol{\varepsilon}( \boldsymbol{v})\|_K^2 + \sum_{F\in \mathcal{F}_h^{K}}h_F^{-1}\|\Pi_k^F[\boldsymbol{v}]\|_F^2, \\
\|\boldsymbol{\tau}\|^2_{0,h} &= \|\boldsymbol{\tau}\|^2 + \sum_{F\in \mathcal{F}_h^{K}}h_F\|\boldsymbol{\tau} \boldsymbol{n}\|^2_F.
\end{align*}

\begin{lemma}[Korn's inequality for piecewise $H^1$ vector functions]
The following discrete Korn's inequality holds for any $\boldsymbol{v}\in H^1(\mathcal{K}_h;\mathbb R^d)$:
\begin{equation}\label{eq:KornineqltypieceH1}
\|\boldsymbol{v}\|^2_{1,h} \eqsim \sum_{K\in\mathcal{K}_h}\| \boldsymbol{\varepsilon}( \boldsymbol{v})\|_K^2 + \sum_{F\in \mathcal{F}_h^{K}}h_F^{-1}\|\Pi_0^F[\boldsymbol{v}]\|_F^2 \eqsim \sum_{K\in\mathcal{K}_h}\| \nabla\boldsymbol{v}\|_K^2 + \sum_{F\in \mathcal{F}_h^{K}}h_F^{-1}\|[\boldsymbol{v}]\|_F^2.
\end{equation}
\end{lemma}

\begin{proof}
By adapting the proof of Lemma 3.2 in \cite{ChenHuHuang2018} and utilizing the decomposition in \eqref{eq:RMdecomp}, we obtain:
\begin{equation*}
\|\boldsymbol{v}\|^2_{1,h} \eqsim \sum_{K\in\mathcal{K}_h}\| \boldsymbol{\varepsilon}( \boldsymbol{v})\|_K^2 + \sum_{F\in \mathcal{F}_h^{K}}h_F^{-1}\|\Pi_0^F[\boldsymbol{v}]\|_F^2.
\end{equation*}
The full equivalence in \eqref{eq:KornineqltypieceH1} then follows directly from the established Korn's inequalities for piecewise $H^1$ vector functions presented in \cite[(1.22)]{Brenner_MC_2004} and \cite[(34)]{ArnoldBrezziMarini2005}.
\end{proof}

\section{The Staggered Discontinuous Galerkin Scheme}\label{sec:SDG}
This section presents the SDG numerical scheme for the Hellinger-Reissner mixed variational formulation of linear elasticity problems. The well-posedness is established by proving continuity, coercivity, and the inf-sup condition. 

\subsection{Variational formulation}
We propose the primal SDG method based on the first-order system \eqref{equ_elasticity} as follows: for $k\geq0$, find $\boldsymbol{\sigma}_h\in \Sigma_h$ and $\boldsymbol{u}_h\in U_h$ such that
\begin{subequations}\label{elassdg}
\begin{align}
\label{elassdg1}
a_h(\boldsymbol{\sigma}_h,\boldsymbol{\tau}_h)+b_h(\boldsymbol{\tau}_h,\boldsymbol{u}_h)&=-\sum_{T\in\mathcal{T}_h}h_T^2(\boldsymbol{f},\div \boldsymbol{\tau}_h)_T\;\;\forall\,\boldsymbol{\tau}_h\in \Sigma_h,
\\
\label{elassdg2}
b_h(\boldsymbol{\sigma}_h,\boldsymbol{v}_h)\qquad\qquad\qquad&=-(\boldsymbol{f},\boldsymbol{v}_h)\qquad\qquad\qquad\;\,\forall\,\boldsymbol{v}_h\in U_h,
\end{align}
\end{subequations}
where the bilinear forms are
\begin{align*}
a_h(\boldsymbol{\sigma}_h,\boldsymbol{\tau}_h)&:=(\mathcal{A}\boldsymbol{\sigma}_h,\boldsymbol{\tau}_h)
+\sum_{T\in\mathcal{T}_h}h_T^2(\div \boldsymbol{\sigma}_h,\div \boldsymbol{\tau}_h)_T
+\sum_{F\in {\mathcal{F}}_h^{T}\setminus {\mathcal{F}}_h^{K}}h_F\langle [\boldsymbol{\sigma}_h\boldsymbol{n}],[\boldsymbol{\tau}_h\boldsymbol{n}]\rangle_{F},\\
b_h(\boldsymbol{\sigma}_h,\boldsymbol{u}_h)&:=-\sum_{K\in\mathcal{K}_h}(\boldsymbol{\sigma}_h,\boldsymbol{\varepsilon}(\boldsymbol{u}_h))_K+\sum_{F\in\mathcal{F}_h^{K}}\langle \boldsymbol{\sigma}_h \boldsymbol{n},[\boldsymbol{u}_h]\rangle_F.
\end{align*}

The Dirichlet boundary condition $\boldsymbol{u}|_{\partial \Omega} = \boldsymbol{0}$ is imposed weakly in the second term of $b_h(\boldsymbol{\sigma}_h,\boldsymbol{u}_h)$ by $\langle \boldsymbol{\sigma}_h \boldsymbol{n}, \boldsymbol{u}_h\rangle_{\partial \Omega}$.
For $k=0$, by \eqref{eq:divSigma0}, the right hand side of \eqref{elassdg1} vanishes, and the  bilinear form $a_h(\cdot,\cdot)$ can be simplified to
$$
a_h(\boldsymbol{\sigma}_h,\boldsymbol{\tau}_h):=(\mathcal{A}\boldsymbol{\sigma}_h,\boldsymbol{\tau}_h)
+\sum_{F\in {\mathcal{F}}_h^{T}\setminus {\mathcal{F}}_h^{K}}h_F\langle [\boldsymbol{\sigma}_h\boldsymbol{n}],[\boldsymbol{\tau}_h\boldsymbol{n}]\rangle_{F}.
$$

From the Cauchy-Schwarz inequality, the trace inequality and the inverse inequality, we get the following lemma.

\begin{lemma}\label{lemma_conti_ab}
The bilinear forms $a_h(\cdot,\cdot)$ and $b_h(\cdot,\cdot)$ are continuous, i.e. 
\begin{align*}
a_h(\boldsymbol{\sigma}_h,\boldsymbol{\tau}_h)&\lesssim \|\boldsymbol{\sigma}_h\|_{0,h}\|\boldsymbol{\tau}_h\|_{0,h} \quad\forall~\boldsymbol{\sigma}_h,\boldsymbol{\tau}_h\in \Sigma_h + H^1(\Omega;\mathbb S),
\\
b_h(\boldsymbol{\sigma}_h,\boldsymbol{u}_h)&\lesssim \|\boldsymbol{\sigma}_h\|_{0,h}\|\boldsymbol{u}_h\|_{1,h} \quad\forall~\boldsymbol{\sigma}_h\in \Sigma_h + H^1(\Omega;\mathbb S), \boldsymbol{u}_h\in U_h + H^1(\Omega;\mathbb R^d).
\end{align*}
\end{lemma}

\subsection{The well-posedness}
We first verify the inf-sup condition of $b_h(\cdot,\cdot)$. The modification of the interior DoFs for $\Sigma_k$ is to impose this condition.
\begin{lemma} \label{eqn_infsup}
We have the discrete inf-sup condition
\begin{equation}\label{eq:discinfsup}
\|\boldsymbol{u}_h\|_{1,h}\lesssim \sup_{\boldsymbol{\sigma}_h\in\Sigma_h}\frac{b_h(\boldsymbol{\sigma}_h,\boldsymbol{u}_h)}{\|\boldsymbol{\sigma}_h\|_{0,h}}\quad\forall~\boldsymbol{u}_h\in U_h.
\end{equation}
\end{lemma}
\begin{proof}
Let $\boldsymbol{\sigma}_h\in\Sigma_h$ satisfy
\begin{align}\label{eqn_sigm1}
\langle \boldsymbol{\sigma}_h\boldsymbol{n},\boldsymbol{q} \rangle_F&=\langle h_F^{-1}\Pi_k^F[\boldsymbol{u}_h], \boldsymbol{q} \rangle_F\quad \forall~\boldsymbol{q}\in   V_{k}(F), F\in \mathcal{F}_h^{K}, \\ \label{eqn_sigm2}
(\boldsymbol{\sigma}_h,\boldsymbol{\tau})_K&=-(\boldsymbol{\varepsilon}( \boldsymbol{u}_h),\boldsymbol{\tau})_K \quad\quad\; \forall~\boldsymbol{\tau}\in \mathbb{P}_{k}(K;\mathbb{S}), K\in\mathcal{K}_h,
\end{align}
and DoFs \eqref{SigmaKDoFs2} vanish. Using the scaling argument, we have
\begin{equation*}
\|\boldsymbol{\sigma}_h\|_{0,h}\lesssim \|\boldsymbol{u}_h\|_{1,h}.
\end{equation*}
Based on \eqref{eqn_sigm1}-\eqref{eqn_sigm2} and the definitions of $b_h(\cdot,\cdot)$ and $\|\cdot\|_{1,h}$, we get
\begin{equation*}
\|\boldsymbol{u}_h\|_{1,h}^2= \sum_{K\in\mathcal{K}_h}\|\boldsymbol{\varepsilon}( \boldsymbol{u}_h)\|_K^2
+\sum_{F\in \mathcal{F}_h^{K}}h_F^{-1}\|\Pi_k[\boldsymbol{u}_h]\|_{F}^2=b_h(\boldsymbol{\sigma}_h,\boldsymbol{u}_h).
\end{equation*}
Therefore, the discrete inf-sup condition \eqref{eq:discinfsup} holds.
\end{proof}

Next we show that the bilinear form $a_h(\cdot,\cdot)$ is uniformly coercive on the kernel space
$$
Z_h:=\left\{ \boldsymbol{\sigma}_h\in \Sigma_h: \tr(\boldsymbol{\sigma}_h)\in L_0^2(\Omega), \textrm{ and } b_h(\boldsymbol{\sigma}_h,\boldsymbol{v}_h)=0 \textrm{ for all } \boldsymbol{v}_h\in U_h\right\}.
$$
In Theorem~\ref{thm:elassdg}, we will prove that $\tr(\boldsymbol{\sigma}_h)\in L_0^2(\Omega)$. Therefore, we include the condition $\tr(\boldsymbol{\sigma}_h)\in L_0^2(\Omega)$ in the definition of $Z_h$.

\begin{lemma}\label{lem:trsigmabound}
For any $\boldsymbol{\sigma}_h\in Z_h$, we have
\begin{equation}\label{eq:trsigmabound}
\|\tr(\boldsymbol{\sigma}_h)\|^2\lesssim \|\dev\boldsymbol{\sigma}_h\|^2+\sum_{T\in\mathcal{T}_h} h_T^2\|\div\boldsymbol{\sigma}_h\|_T^2
+\sum_{F\in \mathcal{F}_h^{T} \setminus \mathcal{F}_h^{K}} h_F\|[\boldsymbol{\sigma}_h \boldsymbol{n}]\|_F^2,
\end{equation}
where the constant is independent of the mesh size $h$ and Lam\'{e} constants.
\end{lemma}
\begin{proof}
Since $\tr(\boldsymbol{\sigma}_h)\in L_0^2(\Omega)$, we have a $\boldsymbol{v}\in H^1_0(\Omega;\mathbb{R}^d)$ satisfying
\begin{equation}\label{eq:202411131}
-\div\boldsymbol{v}=\tr(\boldsymbol{\sigma}_h)\quad \textrm{ and } \quad \|\boldsymbol{v}\|_1 \lesssim \|\tr(\boldsymbol{\sigma}_h)\|.
\end{equation}
Let $\boldsymbol{v}_h= Q_{k+1}\boldsymbol{v}\in U_h$.
Noting that $\boldsymbol{\sigma}_h\in Z_h$, we obtain from the integration by parts and the Cauchy-Schwarz inequality that
\begin{align*}
\frac{1}{d}\|\tr(\boldsymbol{\sigma}_h)\|^2_0&= -\frac{1}{d}(\tr(\boldsymbol{\sigma}_h),\div  \boldsymbol{v}) 
= -\frac{1}{d}(\tr(\boldsymbol{\sigma}_h)\boldsymbol{I}, \boldsymbol{\varepsilon}(\boldsymbol{v})) 
= (\dev\boldsymbol{\sigma}_h-\boldsymbol{\sigma}_h, \boldsymbol{\varepsilon}(\boldsymbol{v})) 
\\
&= (\dev\boldsymbol{\sigma}_h, \boldsymbol{\varepsilon}(\boldsymbol{v})) + b_h(\boldsymbol{\sigma}_h,\boldsymbol{v}) = (\dev\boldsymbol{\sigma}_h, \boldsymbol{\varepsilon}(\boldsymbol{v})) + b_h(\boldsymbol{\sigma}_h,\boldsymbol{v}-\boldsymbol{v}_h)
\\
&= (\dev\boldsymbol{\sigma}_h, \boldsymbol{\varepsilon}(\boldsymbol{v}))
+\sum_{T\in\mathcal{T}_h}(\div \boldsymbol{\sigma}_h, \boldsymbol{v}- \boldsymbol{v}_h)_{T}
-\sum_{F\in \mathcal{F}_h^{T} \setminus \mathcal{F}_h^{K}}\langle  [\boldsymbol{\sigma}_h \boldsymbol{n}],\boldsymbol{v}-\boldsymbol{v}_h\rangle_F 
\\
& \lesssim \|\dev\boldsymbol{\sigma}_h\|\|\boldsymbol{v}\|_1
+ \sum_{T\in\mathcal{T}_h} \|\div \boldsymbol{\sigma}_h\|_T\|\boldsymbol{v}- \boldsymbol{v}_h\|_T
\\
&\quad + \sum_{F\in \mathcal{F}_h^{T} \setminus \mathcal{F}_h^{K}} \|[\boldsymbol{\sigma}_h \boldsymbol{n}]\|_F\|\boldsymbol{v}- \boldsymbol{v}_h\|_F.
\end{align*}
This combined with the error estimate of $Q_{k+1}$ and \eqref{eq:202411131} yields \eqref{eq:trsigmabound}.
\end{proof}

This lemma shows that the trace of the stress can be controlled by the deviatoric part together with appropriate stabilization terms. The subsequent corollary will demonstrate that the coercivity of the numerical scheme is independent of the Lam\'{e} constants, implying the robustness of the proposed scheme.

\begin{corollary}\label{Corol_coercive}
The following coercivity condition holds, 
\begin{equation}\label{eq:ahcoercive}
\|\boldsymbol{\sigma}_h\|_{0,h}^2\lesssim a_h(\boldsymbol{\sigma}_h,\boldsymbol{\sigma}_h)\quad\forall~\boldsymbol{\sigma}_h\in Z_h,
\end{equation}
with a constant independent of the mesh size $h$ and the Lam\'e constants.
\end{corollary}
\begin{proof}
After a simple rearrangement, we get
\begin{eqnarray}\label{eq:lowerBoudA}
(\mathcal{A}\boldsymbol{\sigma}_h,\boldsymbol{\sigma}_h)=\frac{1}{2\mu}\|\dev\boldsymbol{\sigma}_h\|^2 +\frac{1}{d(2\mu+d\lambda)}\|\tr(\boldsymbol{\sigma}_h)\|^2.
\end{eqnarray}
From the conclusion given in Lemma \ref{lem:trsigmabound}, we get
\begin{eqnarray*}
\displaystyle
\|\boldsymbol{\sigma}_h\|^2&=& \|\dev\boldsymbol{\sigma}_h\|^2+\frac{1}{d}\|\tr(\boldsymbol{\sigma}_h)\|^2
\\  [0.08in]
&\lesssim& \|\dev\boldsymbol{\sigma}_h\|^2+\sum_{T\in\mathcal{T}_h} h_T^2 \|\div \boldsymbol{\sigma}_h\|_T^2
+\sum_{F\in \mathcal{F}_h^{T} \setminus \mathcal{F}_h^{K}} h_F\|[\boldsymbol{\sigma}_h \boldsymbol{n}]\|_F^2\lesssim a_h(\boldsymbol{\sigma}_h,\boldsymbol{\sigma}_h).
\end{eqnarray*}
Then \eqref{eq:ahcoercive} holds from the trace inequality and inverse inequality.
\end{proof}

As shown by \eqref{eq:lowerBoudA}, the coercivity of the bilinear form $(\mathcal{A}\boldsymbol{\sigma}_h, \boldsymbol{\sigma}_h)$ is lost in the limit $\lambda \to 0$. Hence, to achieve a stable numerical scheme, a stabilization term
\begin{align*}
\sum_{T\in\mathcal{T}_h}h_T^2(\div \boldsymbol{\sigma}_h,\div \boldsymbol{\tau}_h)_T
+\sum_{F\in {\mathcal{F}}_h^{T}\setminus {\mathcal{F}}_h^{K}}h_F\langle [\boldsymbol{\sigma}_h\boldsymbol{n}],[\boldsymbol{\tau}_h\boldsymbol{n}]\rangle_{F}
\end{align*}
 is added to the bilinear form $a_h(\boldsymbol{\sigma}_h,\boldsymbol{\tau}_h)$. We stress that this stabilizer is specifically designed to prevent numerical locking, and not to handle any continuity issues arising from nonconforming elements. This differs from virtual element methods \cite{BeiraodaVeigaBrezziMarini2013} and weak Galerkin finite element methods \cite{WangWangWangZhang2016}, where the purpose of introducing a stabilizer is to obtain a stable bilinear form, and therefore the definition of their stabilizer depends on the choice of the approximation function space.

\begin{theorem}\label{thm:elassdg}
The SDG method \eqref{elassdg} is well-posed and stable. There exists a unique solution pair $(\boldsymbol{\sigma}_h, \boldsymbol{u}_h)\in \Sigma_h\times U_h$ satisfying $\tr(\boldsymbol{\sigma}_h)\in L_0^2(\Omega)$ and
$$
\|\boldsymbol{\sigma}_h\|_{0,h}+\|\boldsymbol{u}_h\|_{1,h}\lesssim \|\boldsymbol{f}\|.
$$
\end{theorem}
\begin{proof}
We first show the well-posedness. Assume $\boldsymbol{f}=\boldsymbol{0}$.
Testing \eqref{elassdg1} with $\boldsymbol{\tau}_h=\boldsymbol{I}$ gives $\tr(\boldsymbol{\sigma}_h)\in L_0^2(\Omega)$.
From \eqref{elassdg2} we have $\boldsymbol{\sigma}_h\in Z_h$.
Taking $\boldsymbol{\tau}_h=\boldsymbol{\sigma}_h$ in \eqref{elassdg1} yields
\[
a_h(\boldsymbol{\sigma}_h,\boldsymbol{\sigma}_h)=0,
\]
and the discrete coercivity \eqref{eq:ahcoercive} then implies $\boldsymbol{\sigma}_h=\boldsymbol{0}$.
With this, \eqref{elassdg1} reduces to
\[
b_h(\boldsymbol{\tau}_h,\boldsymbol{u}_h)=0 \qquad \forall~\boldsymbol{\tau}_h\in\Sigma_h.
\]
By the discrete inf-sup condition \eqref{eq:discinfsup}, we conclude $\boldsymbol{u}_h=\boldsymbol{0}$.
Therefore, the SDG method \eqref{elassdg} is well-posed.

For a general $\boldsymbol{f}$, testing \eqref{elassdg1} with $\boldsymbol{\tau}_h=\boldsymbol{I}$ again shows $\tr(\boldsymbol{\sigma}_h)\in L_0^2(\Omega)$.
The stability estimate follows by combining the discrete inf-sup condition \eqref{eq:discinfsup} with the discrete coercivity \eqref{eq:ahcoercive}.
\end{proof}

\section{Hybridization}\label{sec:hybrid}
In this section, we will present the hybridization of the SDG method \eqref{elassdg}. If we eliminate the stress variable, we obtain a primal formulation. 

\subsection{Spaces and weak differential operators}
Let
$$
M_h=\{\boldsymbol{u}_h=\{\boldsymbol{u}_0,\boldsymbol{u}_b\}: \; \boldsymbol{u}_0|_K\in \mathbb{P}_{k+1}(K;\mathbb{R}^d), \boldsymbol{u}_b|_F\in   V_{k}(F), \; K\in\mathcal{K}_h, \; F\in \mathcal{F}_h^{K}\},
$$
and
$$
\Sigma^{-1}_h=\{\boldsymbol{\sigma}_h\in L^2(\Omega;\mathbb{S}):\; \boldsymbol{\sigma}_h|_T\in \Sigma_k(T), T\in \mathcal{T}_h\}.
$$
Let $M_h^0$ be a subspace of $M_h$ with vanishing boundary values on $\partial \Omega$, that is
$$
M_h^0=\{\boldsymbol{u}_h=\{\boldsymbol{u}_0,\boldsymbol{u}_b\}\in M_h: \;  \boldsymbol{u}_b|_F=\boldsymbol{0}, F\in \partial \Omega\}.
$$

Define a weak symmetric gradient operator
$
\boldsymbol{\varepsilon}_w: M_h \rightarrow \Sigma^{-1}_h
$
as follows: for $\boldsymbol{u}_h\in M_h$, $\boldsymbol{\varepsilon}_w(\boldsymbol{u}_h)|_K=\boldsymbol{\varepsilon}_{w,K}(\boldsymbol{u}_h)$ for each $K\in\mathcal{K}_h$, where
\begin{equation*}
(\boldsymbol{\varepsilon}_{w,K}(\boldsymbol{u}_h), \boldsymbol{\tau})_K=(\boldsymbol{\varepsilon}(\boldsymbol{u}_0), \boldsymbol{\tau})_K+\langle \boldsymbol{u}_b-\boldsymbol{u}_0, \boldsymbol{\tau}\boldsymbol{n}\rangle_{\partial K}\quad \forall~\boldsymbol{\tau}\in\Sigma_h^{-1}(K).
\end{equation*}
Unlike the weak Galerkin method \cite{WangWangWangZhang2016}, the local space to compute $\boldsymbol{\varepsilon}_w(\boldsymbol{u}_h)$ is enriched from a single polynomial space $\mathbb P_{k}(K; \mathbb S)$ to the discontinuous polynomial space $\Sigma^{-1}_h(K)$. This enrichment eliminates the stabilization. 

Define the weak divergence operator
$
\div _w: \Sigma^{-1}_h \rightarrow M_h
$
by
$$
\div _w \boldsymbol{\sigma}_h=\{\div _{w,K}\boldsymbol{\sigma}_h, -h_F^{-1}[\boldsymbol{\sigma}_h\boldsymbol{n}]\}_{\mathcal{K}_h,\mathcal{F}_h^{K}},
$$
where the local weak divergence $\div _{w,K} \boldsymbol{\sigma}_h \in \mathbb{P}_{k+1}(K; \mathbb{R}^d)$ is defined as
$$
(\div _{w,K}\boldsymbol{\sigma}_h,\boldsymbol{u}_0)_K=
\sum_{T\in\mathcal{T}_h(K)}(\div  \boldsymbol{\sigma}_h, \boldsymbol{u}_0)_T 
-\sum_{F\in \mathcal{F}_h^T(K)\setminus \partial K} \langle  [\boldsymbol{\sigma}_h\boldsymbol{n}], \boldsymbol{u}_0\rangle_{F},
$$
for any $\boldsymbol{u}_0\in \mathbb{P}_{k+1}(K;\mathbb{R}^d)$.
For $\boldsymbol{u}_h, \boldsymbol{v}_h \in M_h$, introduce a discrete $L^2$ inner product
$$
(\boldsymbol{u}_h,\boldsymbol{v}_h)_{0,h}=(\boldsymbol{u}_0,\boldsymbol{v}_0)+\sum_{F\in \mathcal{F}_h^{K}}h_F\langle \boldsymbol{u}_b,\boldsymbol{v}_b\rangle_{F},
$$
and the discrete $L^2$ norm
$$
\|\boldsymbol{u}_h\|_{0,h}^2=(\boldsymbol{u}_h,\boldsymbol{u}_h)_{0,h}.
$$

\begin{lemma}[Adjoint Property of Weak Differential Operators]
Let $M_h$ and $\Sigma_h^{-1}$ be the discrete spaces for the displacement and stress fields, respectively. For any $\boldsymbol{u}_h = \{\boldsymbol{u}_0, \boldsymbol{u}_b\} \in M_h$ and $\boldsymbol{\sigma}_h \in \Sigma_h^{-1}$, the weak strain operator $\boldsymbol{\varepsilon}_w$ and the weak divergence operator $\div_w$ satisfy the following adjoint relation:
\begin{equation}
(\boldsymbol{\varepsilon}_w(\boldsymbol{u}_h), \boldsymbol{\sigma}_h) = -(\boldsymbol{u}_h, \div_w \boldsymbol{\sigma}_h)_{0,h}.
\end{equation}
\end{lemma}

\begin{proof}
By applying the definition of the weak strain operator $\boldsymbol{\varepsilon}_{w,K}$ on each element $K \in \mathcal{K}_h$ and using integration by parts, we have:
\begin{align*}
(\boldsymbol{\varepsilon}_w(\boldsymbol{u}_h), \boldsymbol{\sigma}_h) &= \sum_{K \in \mathcal{K}_h} (\boldsymbol{\varepsilon}_{w,K}(\boldsymbol{u}_h), \boldsymbol{\sigma}_h)_K \\
&= \sum_{K \in \mathcal{K}_h} \left( (\boldsymbol{\varepsilon}(\boldsymbol{u}_0), \boldsymbol{\sigma}_h)_K + \langle \boldsymbol{u}_b - \boldsymbol{u}_0, \boldsymbol{\sigma}_h \boldsymbol{n} \rangle_{\partial K} \right) \\
&= -\sum_{T \in \mathcal{T}_h} (\boldsymbol{u}_0, \div \boldsymbol{\sigma}_h)_T + \sum_{F\in \mathcal{F}_h^T\setminus \mathcal{F}_h} \langle  [\boldsymbol{\sigma}_h\boldsymbol{n}], \boldsymbol{u}_0\rangle_{F} + \sum_{K \in \mathcal{K}_h} \langle \boldsymbol{u}_b, \boldsymbol{\sigma}_h \boldsymbol{n} \rangle_{\partial K}.
\end{align*}
By the definitions of $\div_{w,K}$ and $\div_{w}$,
\begin{align*}
(\boldsymbol{\varepsilon}_w(\boldsymbol{u}_h), \boldsymbol{\sigma}_h) &= -\sum_{K \in \mathcal{K}_h} (\div_{w,K} \boldsymbol{\sigma}_h, \boldsymbol{u}_0)_K + \sum_{F \in \mathcal{F}_h} \langle \boldsymbol{u}_b, [\boldsymbol{\sigma}_h \boldsymbol{n}] \rangle_F = -(\div_w \boldsymbol{\sigma}_h, \boldsymbol{u}_h)_{0,h}.
\end{align*}
This completes the proof that the weak gradient (strain) and weak divergence are adjoint operators in the discrete setting.
\end{proof}

\subsection{The hybridized formulation}
The hybridized formulation of the SDG method \eqref{elassdg}: for $k\geq0$, find $\boldsymbol{\sigma}_h\in \Sigma_h^{-1}$ and $\boldsymbol{u}_h\in M_h^0$ such that
\begin{subequations}\label{elasshdg}
\begin{align}
\label{elasshdg1}
a_h(\boldsymbol{\sigma}_h,\boldsymbol{\tau}_h)+b_h(\boldsymbol{\tau}_h,\boldsymbol{u}_h)&=-\sum_{T\in\mathcal{T}_h}h_T^2(\boldsymbol{f},\div \boldsymbol{\tau}_h)_T\;\;\forall~\boldsymbol{\tau}_h\in \Sigma_h^{-1},
\\
\label{elasshdg2}
b_h(\boldsymbol{\sigma}_h,\boldsymbol{v}_h)\qquad\qquad\qquad&=-(\boldsymbol{f},\boldsymbol{v}_0)\;\qquad\qquad\qquad\forall~\boldsymbol{v}_h\in M_h^0,
\end{align}
\end{subequations}
where
\begin{align*}
a_h(\boldsymbol{\sigma}_h,\boldsymbol{\tau}_h)&:=(\mathcal{A}\boldsymbol{\sigma}_h,\boldsymbol{\tau}_h)
+\sum_{T\in\mathcal{T}_h}h_T^2(\div \boldsymbol{\sigma}_h,\div \boldsymbol{\tau}_h)_T
+\sum_{F\in {\mathcal{F}}_h^{T}\setminus {\mathcal{F}}_h^{K}}h_F\langle [\boldsymbol{\sigma}_h\boldsymbol{n}],[\boldsymbol{\tau}_h\boldsymbol{n}]\rangle_{F},
\\
b_h(\boldsymbol{\sigma}_h,\boldsymbol{u}_h)&:=(\div _w\boldsymbol{\sigma}_h,\boldsymbol{u}_h)_{0,h}=\!-\!\!\sum_{K\in\mathcal{K}_h}(\boldsymbol{\sigma}_h,\boldsymbol{\varepsilon}(\boldsymbol{u}_0))_K+\sum_{K\in\mathcal{K}_h}\langle \boldsymbol{\sigma}_h \boldsymbol{n},\boldsymbol{u}_0-\boldsymbol{u}_b\rangle_{\partial K}.
\end{align*}

We respectively equip spaces $M_h$ and $\Sigma_h^{-1}$ with semi-norms
\begin{align*}
\|\boldsymbol{u}_h\|_{1,h}^2&:=\sum_{K\in\mathcal{K}_h}\|\boldsymbol{\varepsilon}(\boldsymbol{u}_0)\|_K^2+\sum_{K\in\mathcal{K}_h}h_K^{-1}\|\Pi_k^F(\boldsymbol{u}_0-\boldsymbol{u}_b)\|_{\partial K}^2, \qquad\qquad\;\; \boldsymbol{u}_h\in M_h, \\
\|\boldsymbol{\sigma}_h\|_{0,h}^2&:=\|\boldsymbol{\sigma}_h\|^2
+\sum_{K\in \mathcal{K}_h}h_K\|\boldsymbol{\sigma}_h \boldsymbol{n}\|^2_{\partial K}+\sum_{F\in {\mathcal{F}}_h^{T}\setminus {\mathcal{F}}_h^{K}}h_F\|[\boldsymbol{\sigma}_h\boldsymbol{n}]\|_{F}^2,
\qquad \boldsymbol{\sigma}_h\in \Sigma_h^{-1}.
\end{align*}
It is easy to prove that
$\|\cdot\|_{1,h}$ is a norm on space $M_h^0$
and $\|\cdot\|_{0,h}$ is a norm on space $\Sigma_h^{-1}$ with $k\geq 0$.
We have the continuity
\begin{align*}
&a_h(\boldsymbol{\sigma}_h,\boldsymbol{\tau}_h)\lesssim \|\boldsymbol{\sigma}_h\|_{0,h}\|\boldsymbol{\tau}_h\|_{0,h} \quad\forall~\boldsymbol{\sigma}_h,\boldsymbol{\tau}_h\in \Sigma_h^{-1},
\\
&b_h(\boldsymbol{\sigma}_h,\boldsymbol{u}_h)\lesssim \|\boldsymbol{\sigma}_h\|_{0,h}\|\boldsymbol{u}_h\|_{1,h} \quad\forall~\boldsymbol{\sigma}_h\in \Sigma_h^{-1}, \boldsymbol{u}_h\in M_h^0.
\end{align*}

\begin{lemma}
We have the discrete inf-sup condition
\begin{equation}\label{eq:hdgdiscinfsup}
\|\boldsymbol{u}_h\|_{1,h}\lesssim \sup_{\boldsymbol{\sigma}_h\in\Sigma_h^{-1}}\frac{b_h(\boldsymbol{\sigma}_h,\boldsymbol{u}_h)}{\|\boldsymbol{\sigma}_h\|_{0,h}}\quad\forall~\boldsymbol{u}_h\in M_h^0.
\end{equation}
\end{lemma}

\begin{proof}
For any given $\boldsymbol{u}_h\in M_h^0$, let $\boldsymbol{\sigma}_h\in\Sigma^{-1}_h$ satisfy
\begin{align}\label{eqn_sigm3}
\langle \boldsymbol{\sigma}_h\boldsymbol{n},\boldsymbol{q} \rangle_F&=\langle h_K^{-1}\Pi_k^F(\boldsymbol{u}_0-\boldsymbol{u}_b), \boldsymbol{q} \rangle_F\quad \forall~\boldsymbol{q}\in   V_{k}(F), F\in\partial K, \\ \label{eqn_sigm4}
(\boldsymbol{\sigma}_h,\boldsymbol{\tau})_K&=-(\boldsymbol{\varepsilon}( \boldsymbol{u}_0),\boldsymbol{\tau})_K \qquad\qquad\;\;\, \forall~\boldsymbol{\tau}\in \mathbb{P}_{k}(K;\mathbb{S}),
\end{align}
and DoFs \eqref{SigmaKDoFs2} vanish on each $K\in\mathcal{K}_h$. Using the scaling argument, we have
\begin{equation*}
\|\boldsymbol{\sigma}_h\|_{0,h}\lesssim \|\boldsymbol{u}_h\|_{1,h}.
\end{equation*}
Based on \eqref{eqn_sigm3}-\eqref{eqn_sigm4} and the definitions of $b_h(\cdot,\cdot)$ and $\|\cdot\|_{1,h}$, we get
\begin{equation*}
b_h(\boldsymbol{\sigma}_h,\boldsymbol{u}_h) = \sum_{K\in\mathcal{K}_h}\|\boldsymbol{\varepsilon}(\boldsymbol{u}_0)\|_K^2+\sum_{K\in\mathcal{K}_h}h_K^{-1}\| \Pi_k^F(\boldsymbol{u}_0-\boldsymbol{u}_b)\|_{\partial K}^2=\|\boldsymbol{u}_h\|_{1,h}^2.
\end{equation*}
Therefore, the discrete inf-sup condition \eqref{eq:hdgdiscinfsup} follows.
\end{proof}

\begin{lemma}
We have the discrete coercivity
\begin{equation}\label{eq:hdgahcoercive}
\|\boldsymbol{\sigma}_h\|_{0,h}^2\lesssim a_h(\boldsymbol{\sigma}_h,\boldsymbol{\sigma}_h)\quad\forall~\boldsymbol{\sigma}_h\in Z_h,
\end{equation}
where
$$
Z_h:=\left\{ \boldsymbol{\sigma}_h\in \Sigma_h^{-1}: \tr(\boldsymbol{\sigma}_h)\in L_0^2(\Omega), \textrm{ and } b_h(\boldsymbol{\sigma}_h,\boldsymbol{v}_h)=0 \textrm{ for all } \boldsymbol{v}_h\in M_h^0\right\}.
$$
\end{lemma}
\begin{proof}
By choosing $\boldsymbol{v}_h=\{0,\boldsymbol{v}_b\}\in M_h^0$ in the definition of $Z_h$, we find that $Z_h\subseteq \Sigma_h$.
Thus, we end the proof by applying the discrete coercivity \eqref{eq:ahcoercive}.
\end{proof}

\begin{theorem}
Let $\boldsymbol{\sigma}_h\in \Sigma_h^{-1}$ and $\boldsymbol{u}_h=\{\boldsymbol{u}_0,\boldsymbol{u}_b\}\in M_h^0$ be the solution of the hybridized formulation \eqref{elasshdg}. 
The hybridized SDG method \eqref{elasshdg} is well-posed and stable. 
The following stability estimate holds:
\begin{equation}\label{eq:hybridstability}
\|\boldsymbol{\sigma}_h\|_{0,h}+\|\boldsymbol{u}_h\|_{1,h}\lesssim\sup_{\boldsymbol{\tau}_h\in \Sigma_h^{-1},\boldsymbol{v}_h\in M_h^0}\frac{a_h(\boldsymbol{\sigma}_h,\boldsymbol{\tau}_h)+b_h(\boldsymbol{\tau}_h,\boldsymbol{u}_h)+b_h(\boldsymbol{\sigma}_h,\boldsymbol{v}_h)}{\|\boldsymbol{\tau}_h\|_{0,h}+\|\boldsymbol{v}_h\|_{1,h}}.
\end{equation}
Then $\boldsymbol{\sigma}_h\in \Sigma_h$, and the pair $(\boldsymbol{\sigma}_h,\boldsymbol{u}_0)$ is exactly the solution of the SDG method \eqref{elassdg}.
\end{theorem}
\begin{proof}
By the same argument as in Theorem~\ref{thm:elassdg}, the hybridized SDG method \eqref{elasshdg} is well-posed.
The stability estimate \eqref{eq:hybridstability} follows directly from the discrete inf-sup condition \eqref{eq:hdgdiscinfsup} and the discrete coercivity \eqref{eq:hdgahcoercive}.

To show that $\boldsymbol{\sigma}_h\in \Sigma_h$, take $\boldsymbol{v}_h=\{\boldsymbol{0},\boldsymbol{v}_b\}$ in \eqref{elasshdg2} to obtain
\[
\sum_{K\in\mathcal{K}_h}
\langle \boldsymbol{\sigma}_h \boldsymbol{n}, \boldsymbol{v}_b\rangle_{\partial K}
=0
\qquad
\forall\, \boldsymbol{v}_b\in \prod_{F\in\mathcal{F}_h^{K}}  V_{k}(F).
\]
This is precisely the condition ensuring $\boldsymbol{\sigma}_h\in \Sigma_h$.

Finally, taking $\boldsymbol{v}_h=\{\boldsymbol{v}_0,\boldsymbol{0}\}$ in \eqref{elasshdg} eliminates the facet term and reduces the hybridized formulation to the SDG method \eqref{elassdg}.
Thus $(\boldsymbol{\sigma}_h,\boldsymbol{u}_0)$ is exactly the SDG solution.
\end{proof}

\subsection{Error estimates}
In order to conduct error analysis, we introduce nodal interpolations.
Let 
$\boldsymbol{Q}_h=\{\boldsymbol{Q}_{k+1},\Pi_k^F\}$ be the $L^2-$projection operator from $H^1(\Omega;\mathbb R^d)$ onto spaces $M_h$. 
For $\boldsymbol{\sigma}\in H^1(\Omega;\mathbb S)$, let $\boldsymbol{\sigma}_{I}\in \Sigma_h$ be the nodal interpolation based on DoFs~\eqref{SigmaKDoFs} and \eqref{SigmaKDoFsk0}.
We have
\begin{align}
\label{eq:bhIh}
b_h(\boldsymbol{\sigma}-\boldsymbol{\sigma}_{I},\boldsymbol{v})&=\sum_{K\in\mathcal{K}_h}\langle (\boldsymbol{\sigma}-\boldsymbol{\sigma}_{I}) \boldsymbol{n},\boldsymbol{v}_0-\boldsymbol{v}_b\rangle_{\partial K}\quad\forall~\boldsymbol{\sigma}\in H^{1}(\Omega;\mathbb S), \boldsymbol{v}\in M_h^0, \\
\label{eqn_inq_sigma}
\|\boldsymbol{\sigma}-\boldsymbol{\sigma}_{I}\|_{0,h}&\lesssim h^{k+1}\|\boldsymbol{\sigma}\|_{k+1}\qquad\qquad\qquad\qquad\;\forall~\boldsymbol{\sigma}\in H^{k+1}(\Omega;\mathbb S).
\end{align}

\begin{theorem}
Let $\boldsymbol{\sigma}\in H^{k+1}(\Omega; \mathbb{S})$ and $\boldsymbol{u}\in H^{k+2}(\Omega; \mathbb{R}^d)$ be the solution of the linear elasticity \eqref{equ_elasticity}, and let $\boldsymbol{\sigma}_h\in \Sigma_h^{-1}$ and $\boldsymbol{u}_h\in M_h^0$ be the solution of the hybridized method \eqref{elasshdg}. We have the following error estimate
\begin{equation}
\begin{array}{l}
  \displaystyle
\|\boldsymbol{u}-\boldsymbol{u}_h\|_{1,h}+\|\boldsymbol{\sigma}-\boldsymbol{\sigma}_h\|_{0,h}\lesssim h^{k+1}(\|\boldsymbol{\sigma}\|_{k+1}+\|\boldsymbol{u}\|_{k+2}).
\end{array}\label{eqn_H1errest}
\end{equation}
\end{theorem}

\begin{proof}

From the continuous problem \eqref{equ_mixed_formulation}, for any $\boldsymbol{\tau}_h\in\Sigma_h^{-1}$ and $\boldsymbol{v}_h\in M_h^0$ we have

\begin{align*}
a_h(\boldsymbol{\sigma},\boldsymbol{\tau}_h)+b_h(\boldsymbol{\tau}_h,\boldsymbol{Q}_h\boldsymbol{u})&=-\sum_{T\in\mathcal{T}_h}h_T^2(\boldsymbol{f},\div \boldsymbol{\tau}_h)_T + \sum_{K\in\mathcal{K}_h}(\boldsymbol{\tau}_h,\boldsymbol{\varepsilon}(\boldsymbol{u}-\boldsymbol{Q}_{k+1}\boldsymbol{u}))_K\\
\notag
&\quad\; - \sum_{K\in\mathcal{K}_h}(\boldsymbol{\tau}_h\boldsymbol{n},\boldsymbol{u}-\boldsymbol{Q}_{k+1}\boldsymbol{u})_{\partial K},
\\
b_h(\boldsymbol{\sigma},\boldsymbol{v}_h)&=-(\boldsymbol{f},\boldsymbol{v}_0).
\end{align*}
Subtracting the discrete equations \eqref{elasshdg1}-\eqref{elasshdg2} from these yields the error system
\begin{subequations}\label{elasshdgerr}
\begin{align}
\label{elasshdgerr1}
a_h(\boldsymbol{\sigma}-\boldsymbol{\sigma}_h,\boldsymbol{\tau}_h)+b_h(\boldsymbol{\tau}_h,\boldsymbol{Q}_h\boldsymbol{u}-\boldsymbol{u}_h)&=\sum_{K\in\mathcal{K}_h}(\boldsymbol{\tau}_h,\boldsymbol{\varepsilon}(\boldsymbol{u}-\boldsymbol{Q}_{k+1}\boldsymbol{u}))_K
\\
\notag
&\quad\; - \sum_{K\in\mathcal{K}_h}(\boldsymbol{\tau}_h\boldsymbol{n},\boldsymbol{u}-\boldsymbol{Q}_{k+1}\boldsymbol{u})_{\partial K},
\\
\label{elasshdgerr2}
b_h(\boldsymbol{\sigma}_I-\boldsymbol{\sigma}_h,\boldsymbol{v}_h)\qquad\qquad\qquad\qquad\;\;&=\sum_{K\in\mathcal{K}_h}\langle (\boldsymbol{\sigma}_{I}-\boldsymbol{\sigma}) \boldsymbol{n},\boldsymbol{v}_0-\boldsymbol{v}_b\rangle_{\partial K}
\end{align}
\end{subequations}
for any $\boldsymbol{\tau}_h\in \Sigma_h^{-1}$ and $\boldsymbol{v}_h\in M_h^0$. We have used \eqref{eq:bhIh} for \eqref{elasshdgerr2}.

The sum of \eqref{elasshdgerr1} and \eqref{elasshdgerr2} gives
\begin{align*}
&\quad a_h(\boldsymbol{\sigma}_I-\boldsymbol{\sigma}_h,\boldsymbol{\tau}_h)+b_h(\boldsymbol{\tau}_h,\boldsymbol{Q}_h\boldsymbol{u}-\boldsymbol{u}_h)+b_h(\boldsymbol{\sigma}_I-\boldsymbol{\sigma}_h,\boldsymbol{v}_h) \\
&=a_h(\boldsymbol{\sigma}_I-\boldsymbol{\sigma},\boldsymbol{\tau}_h) + \sum_{K\in\mathcal{K}_h}(\boldsymbol{\tau}_h,\boldsymbol{\varepsilon}(\boldsymbol{u}-\boldsymbol{Q}_{k+1}\boldsymbol{u}))_K
- \sum_{K\in\mathcal{K}_h}(\boldsymbol{\tau}_h\boldsymbol{n},\boldsymbol{u}-\boldsymbol{Q}_{k+1}\boldsymbol{u})_{\partial K} \\
&\quad\; + \sum_{K\in\mathcal{K}_h}\langle (\boldsymbol{\sigma}_{I}-\boldsymbol{\sigma}) \boldsymbol{n},\boldsymbol{v}_0-\boldsymbol{v}_b\rangle_{\partial K}.
\end{align*}
Applying the stability estimate \eqref{eq:hybridstability}, the approximation property \eqref{eqn_inq_sigma}, and the standard estimate for the projection $Q_{k+1}$ gives

\begin{equation*}
\|\boldsymbol{\sigma}_I-\boldsymbol{\sigma}_h\|_{0,h}+\|\boldsymbol{Q}_h\boldsymbol{u}-\boldsymbol{u}_h\|_{1,h}\lesssim h^{k+1}(\|\boldsymbol{\sigma}\|_{k+1}+\|\boldsymbol{u}\|_{k+2}).
\end{equation*}
Finally, the triangle inequality and the approximation properties of $\boldsymbol{\sigma}_I$ and $\boldsymbol{Q}_h$ yield \eqref{eqn_H1errest}.
\end{proof}

Next, we consider the $L^2$-error estimate for $\|\boldsymbol{u}-\boldsymbol{u}_0\|$. Consider the dual problem
\begin{equation}
\left\{
\begin{aligned}
-\div\tilde{\boldsymbol{\sigma}}&=\boldsymbol{u}-\boldsymbol{u}_0 \quad \text{in} \;\Omega ,
\\
\mathcal{A}\tilde{\boldsymbol{\sigma}}&=\boldsymbol{\varepsilon}(\tilde{\boldsymbol{u}}) \quad\quad \text{in} \;\Omega,
\\
\tilde{\boldsymbol{u}}&=\boldsymbol{0} \qquad\quad\;\; \text{on} \;\partial\Omega.\label{eq:dualproblem}
\end{aligned}
\right.
\end{equation}
Assume the dual problem \eqref{eq:dualproblem} satisfies the regularity estimate
\begin{equation}\label{eq:dualregularity}
\|\tilde{\boldsymbol{\sigma}}\|_1+\|\tilde{\boldsymbol{u}}\|_2 \lesssim \|\boldsymbol{u}-\boldsymbol{u}_0\|.
\end{equation} 
The proof is standard and thus skipped here. 
\begin{theorem}
Let $\boldsymbol{\sigma}\in H^{k+1}(\Omega; \mathbb{S})$ and $\boldsymbol{u}\in H^{k+2}(\Omega; \mathbb{R}^d)$ be the solution of the linear elasticity \eqref{equ_elasticity}, and $\boldsymbol{\sigma}_h\in \Sigma_h^{-1}$ and $\boldsymbol{u}_h\in M_h^0$ be the solution of the hybridized method \eqref{elasshdg}.
Assume the regularity~\eqref{eq:dualregularity} holds.
We have 
\begin{equation}\label{eq:uL2error}
\|\boldsymbol{u}-\boldsymbol{u}_0\|\lesssim h^{k+2}(\|\boldsymbol{\sigma}\|_{k+1}+\|\boldsymbol{u}\|_{k+2}).
\end{equation}
\end{theorem}

\section{Numerical Examples}\label{sec:numerexam}
In this section, numerical experiments are presented to illustrate the methods given in this paper. 
The numerical experiments are developed based on the MATLAB package $i$FEM \cite{Chen_iFem}.

\subsection{Triangle mesh}
We consider the unit square domain $\Omega=(0,1)^2$ and apply a triangular mesh discretization $\mathcal{K}_h$; see Fig.~\ref{fig:triangulation1} for the mesh size $h=0.25$. By connecting the barycenters and the vertices of each triangle, we obtain a refined triangular mesh $\mathcal{T}_h$;
see Fig.~\ref{fig:triangulation2}. 

\begin{figure}[htbp]
\subfigure[A triangle partition $\mathcal{K}_h$.]{
\begin{minipage}[t]{0.45\linewidth}
\centering
\includegraphics*[width=4.8cm]{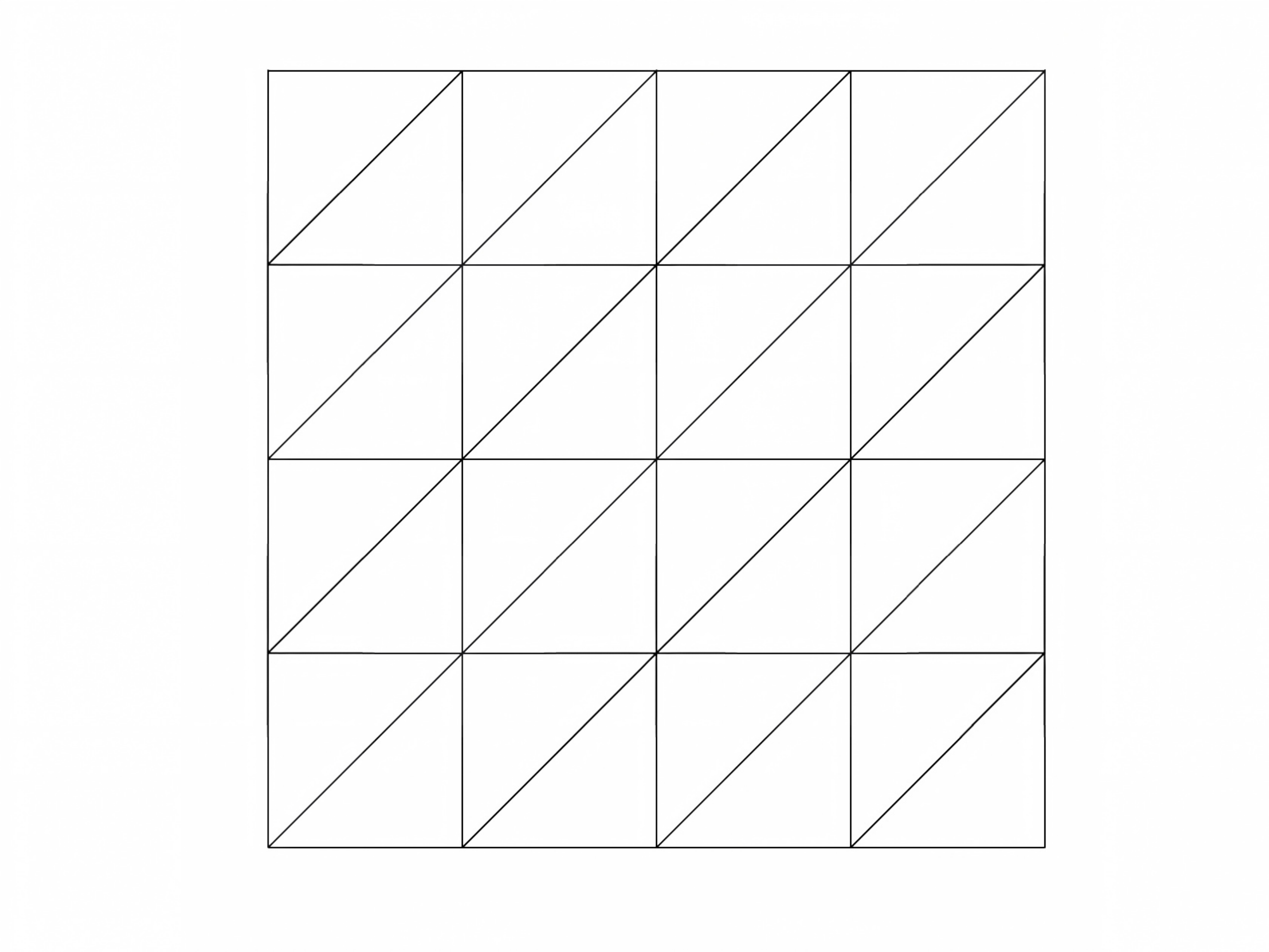}\label{fig:triangulation1}
\end{minipage}}
\subfigure[Triangulation $\mathcal{T}_h$ for $\mathcal{K}_h$.]
{\begin{minipage}[t]{0.45\linewidth}
\centering
\includegraphics*[width=4.8cm]{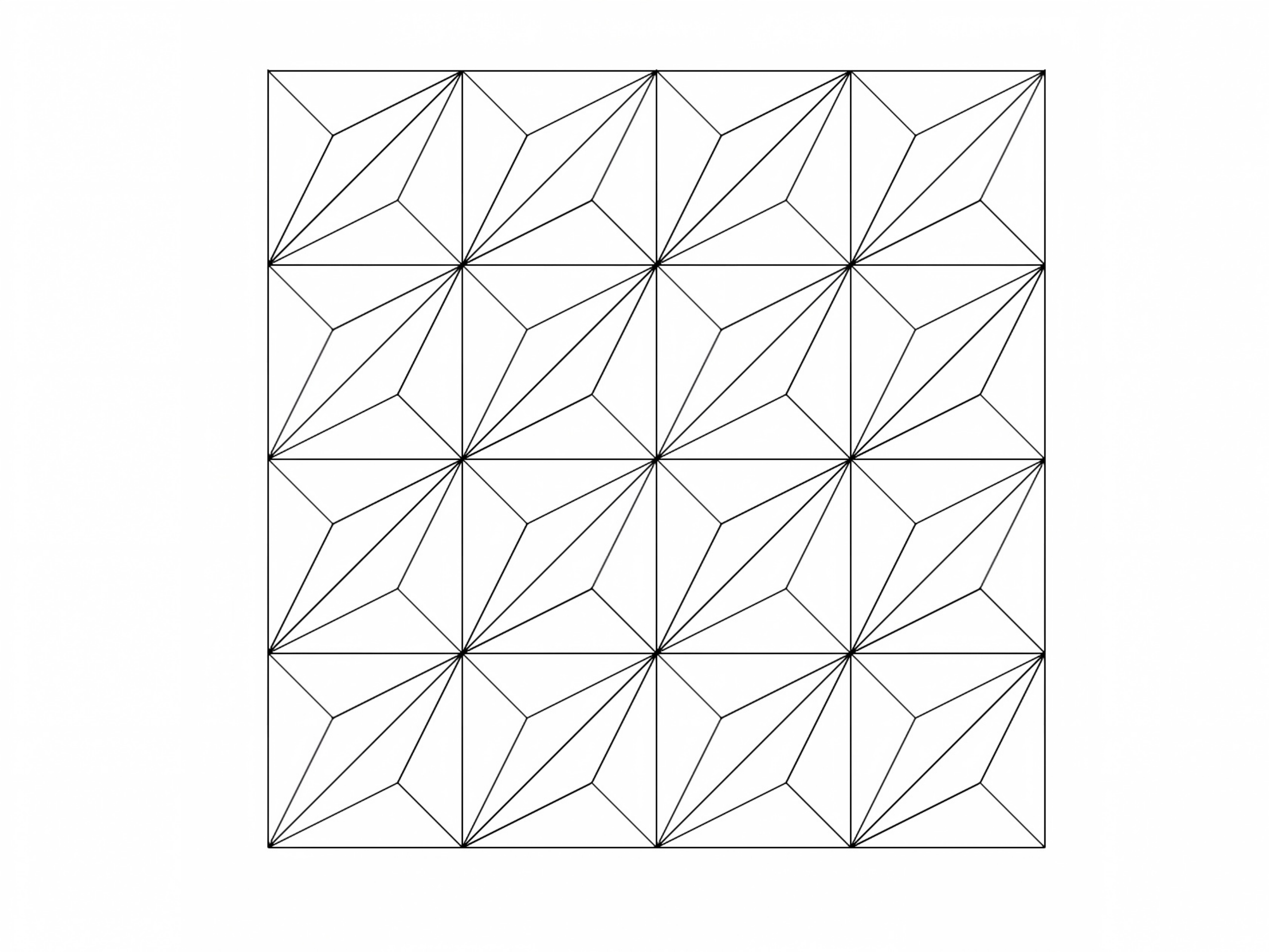}\label{fig:triangulation2}
\end{minipage}}
\caption{Partitions of $\Omega$ for $h=0.25$.}
\label{fig:triangulation}
\end{figure}

In this example, we examine the locking-free property of the numerical scheme \eqref{elasshdg}. As the parameter $\lambda$ increases, the numerical error of the scheme should not exhibit growth that depends on $\lambda$.
We consider the exact solution and the right-hand side is correspondingly given by
\begin{equation}\label{eqn_NM_realu}
\mathbf{u}=\left(\begin{array}{l}\cos(\pi x) \cos(\pi y)\\ 
\sin(\pi x) \sin(\pi y)
\end{array}\right)
\quad 
f=2\mu\pi^2 \left(\begin{array}{l}\cos(\pi x) \cos(\pi y)\\ 
\sin(\pi x) \sin(\pi y)
\end{array}\right).
\end{equation}

Table \ref{Ex1_LockingFree} presents the errors and convergence rates between the numerical solution of \eqref{elasshdg} in the lowest order case $k=0$ and
the real solution \eqref{eqn_NM_realu}, with $\mu=1$ and $\lambda=1, 10^2, 10^4, 10^6$. 
The norms are given by 
\begin{align*}
\|\boldsymbol{u}-\boldsymbol{u}_0\|^2&=\sum_{K\in\mathcal{K}_h}\|\boldsymbol{u}-\boldsymbol{u}_0\|_K^2, 
\\
\|\boldsymbol{u}-\boldsymbol{u}_h\|_{1,h}^2&=\sum_{K\in\mathcal{K}_h}\|\boldsymbol{\varepsilon}(\boldsymbol{u}-\boldsymbol{u}_0)\|_K^2
+\sum_{K\in\mathcal{K}_h}h_K^{-1}\|\Pi_k^F\boldsymbol{u}_0-\boldsymbol{u}_b\|_{\partial K}^2, 
\\
\|\boldsymbol{\sigma}-\boldsymbol{\sigma}_h\|^2_{0,h}&=\sum_{K\in\mathcal{K}_h}\|\boldsymbol{\sigma}-\boldsymbol{\sigma}_h\|_K^2+\sum_{F\in {\mathcal{F}}_h^{T}\setminus {\mathcal{F}}_h^{K}}h_F\|[\boldsymbol{\sigma}_h\boldsymbol{n}]\|_{F}^2.
\end{align*}
The data in the table confirm that the numerical error remains independent of the parameter $\lambda$.

\begin{table}[!ht]
  \centering
  \caption{Example 1 (Locking free):
    Error and the convergence rate. }
    \renewcommand{\arraystretch}{1.125}
    \resizebox{11.5cm}{!}{
  \begin{tabular}{@{}c c c c c c c c @{}}
    \toprule
  $ h $  &    $N_K$   &  $ \|\boldsymbol{u}-\boldsymbol{u}_0\| $  
  &  Rate   &   $ \|\boldsymbol{u}-\boldsymbol{u}_h\|_{1,h} $  &  Rate &
  $ \|\boldsymbol{\sigma}-\boldsymbol{\sigma}_h\|_{0,h} $ & Rate\\
    \hline
    \multicolumn{8}{c}{$\lambda=1$}\\
    \hline
1.250e-01  &   128 &  4.37529e-02 &  -- &  1.25560e+00 &  -- &  1.24368e+00 &  --\\
6.250e-02  &   512 &  1.09062e-02 &  2.00 &  6.18037e-01 &  1.02 &  4.85307e-01 &  1.36\\
3.125e-02  &  2048 &  2.72453e-03 &  2.00 &  3.07778e-01 &  1.01 &  2.07986e-01 &  1.22\\
1.562e-02  &  8192 &  6.81071e-04 &  2.00 &  1.53732e-01 &  1.00 &  9.52801e-02 &  1.13\\
7.812e-03  & 32768 &  1.70272e-04 &  2.00 &  7.68458e-02 &  1.00 &  4.54582e-02 &  1.07\\
    \hline
    \multicolumn{8}{c}{$\lambda=10^2$}\\
    \hline
1.250e-01 &    128 &  4.01417e-02  &  --  &  1.24811e+00 &  -- &  1.25482e+00  & --\\
6.250e-02 &    512 &  1.00279e-02  &  2.00 &  6.16168e-01 &  1.02 &  4.90747e-01 &  1.35\\
3.125e-02 &   2048 &  2.50798e-03  &  2.00 &  3.07038e-01 &  1.00 &  2.10577e-01 &  1.22\\
1.562e-02 &   8192 &  6.27143e-04  &  2.00 &  1.53382e-01 &  1.00 &  9.65415e-02 &  1.13\\
7.812e-03 &  32768 &  1.56797e-04  &  2.00 &  7.66733e-02 &  1.00 &  4.60843e-02 &  1.07\\
    \hline
    \multicolumn{8}{c}{$\lambda=10^4$}\\
    \hline
1.250e-01  &   128 &  4.00644e-02 &  -- &  1.24802e+00 &  -- &  1.25524e+00 &  --\\
6.250e-02  &   512 &  1.00140e-02 &  2.00 &  6.16165e-01 &  1.02 &  4.90943e-01 &  1.35\\
3.125e-02  &  2048 &  2.50543e-03 &  2.00 &  3.07040e-01 &  1.00 &  2.10661e-01 &  1.22\\
1.562e-02  &  8192 &  6.26628e-04 &  2.00 &  1.53384e-01 &  1.00 &  9.65786e-02 &  1.13\\
7.812e-03  & 32768 &  1.56683e-04 &  2.00 &  7.66740e-02 &  1.00 &  4.61017e-02 &  1.07\\
    \hline
    \multicolumn{8}{c}{$\lambda=10^6$}\\
    \hline
1.250e-01 &    128 &  4.00636e-02 &  -- &  1.24801e+00 &  -- &  1.25524e+00 &  --\\
6.250e-02 &    512 &  1.00138e-02 &  2.00 &  6.16165e-01 &  1.02 &  4.90945e-01 &  1.35\\
3.125e-02 &   2048 &  2.50540e-03 &  2.00 &  3.07040e-01 &  1.00 &  2.10662e-01 &  1.22\\
1.562e-02 &   8192 &  6.26623e-04 &  2.00 &  1.53384e-01 &  1.00 &  9.65790e-02 &  1.13\\
7.812e-03 &  32768 &  1.56682e-04 &  2.00 &  7.66740e-02 &  1.00 &  4.61019e-02 &  1.07\\
    \bottomrule
  \end{tabular}}
\label{Ex1_LockingFree}
\end{table}

\subsection{Polygon mesh}
In this example, we consider the polygon meshes on the unit square domain $\Omega = (0,1)^2$. Fig.~\ref{fig_polygonmesh1} illustrates the partition into $16$ polygonal elements.
By connecting the barycenters and vertices of each polygon, we obtain a refined triangular mesh $\mathcal{T}_h$, as shown in Fig.~\ref{fig_polygonmesh2}.
The exact solution and the right-hand side are chosen to be the same as in Example 1.
Let $\mu = 1$ and $\lambda = 1$. The errors and convergence rates are presented in Table \ref{Ex2_P0P1poly} for $k = 0$ and in Table \ref{Ex2_P1P2poly} for $k = 1$.
We observe convergence orders of $k+1$ for $\|\boldsymbol{u}-\boldsymbol{u}_h\|_{1,h}$ and $\|\boldsymbol{\sigma}-\boldsymbol{\sigma}_h\|_{0,h}$, and $k+2$ for $ \|\boldsymbol{u}-\boldsymbol{u}_0\| $, which are consistent with our theoretical analysis.

\begin{figure}[htbp]
\subfigure[A polygon partition $\mathcal{K}_h$.]{
\begin{minipage}[t]{0.45\linewidth}
\centering
\includegraphics*[width=4.8cm]{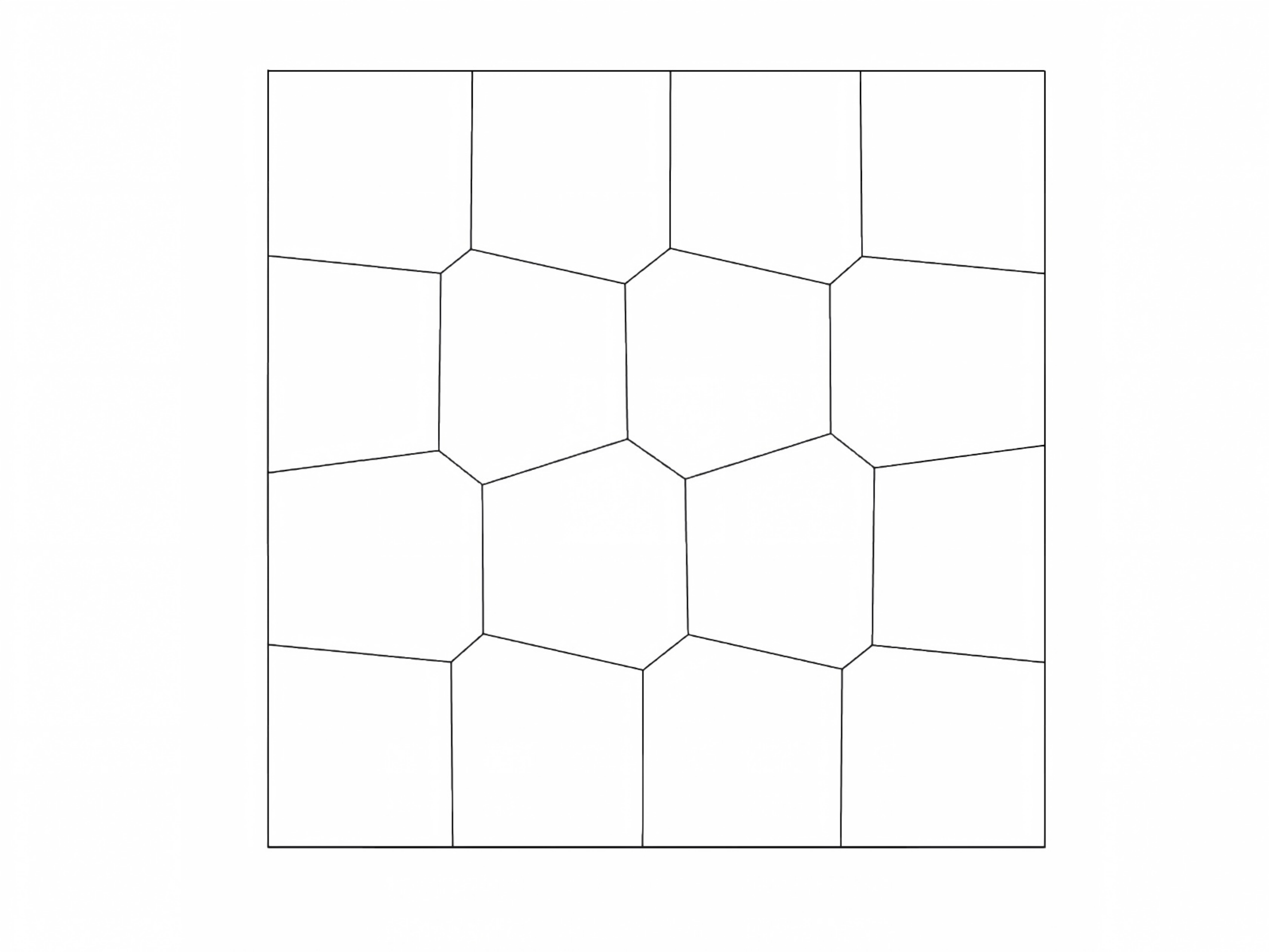}\label{fig_polygonmesh1}
\end{minipage}}
\subfigure[Triangulation $\mathcal{T}_h$ for $\mathcal{K}_h$.]
{\begin{minipage}[t]{0.45\linewidth}
\centering
\includegraphics*[width=4.8cm]{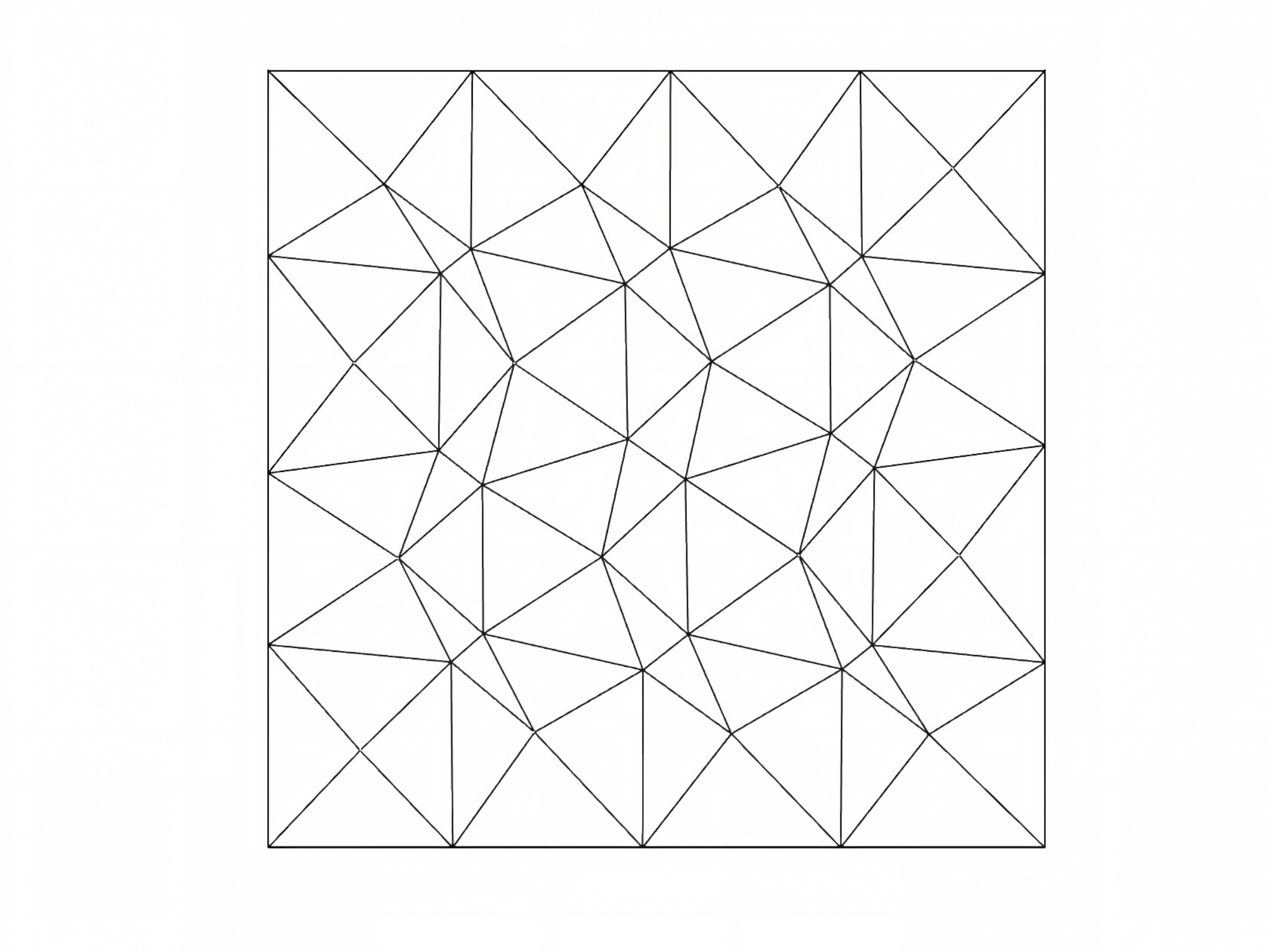}\label{fig_polygonmesh2}
\end{minipage}}
\caption{A polygonal mesh of $\Omega$ with $N_K=16$.}
\label{fig_polygonmesh}
\end{figure}

\begin{table}[!ht]
  \centering
  \caption{Example 2 (Polygon mesh):
    Error and the convergence rate for $k=0$.}
    \renewcommand{\arraystretch}{1.125}
    \resizebox{11.5cm}{!}{
  \begin{tabular}{@{}c c c c c c c c @{}}
    \toprule
  $ h $   &    $ N_K$ & $ \|\boldsymbol{u}-\boldsymbol{u}_0\| $  
  &  Rate   &   $ \|\boldsymbol{u}-\boldsymbol{u}_h\|_{1,h} $  &  Rate &
  $ \|\boldsymbol{\sigma}-\boldsymbol{\sigma}_h\|_{0,h} $ & Rate
    \\
    \hline
1.250e-01 &     64 &  5.22450e-02 &  -- &  1.13788e+00 &  -- &  9.64047e-01 &  --\\
6.250e-02 &    256 &  1.21551e-02 &  2.10 &  5.11926e-01 &  1.15 &  4.61249e-01 &  1.06\\
3.125e-02 &   1024 &  2.99231e-03 &  2.02 &  2.52590e-01 &  1.02 &  2.31233e-01 &  1.00\\
1.562e-02 &   4096 &  7.42393e-04 &  2.01 &  1.24521e-01 &  1.02 &  1.14865e-01 &  1.01\\
7.069e-03 &  20014 &  1.83153e-04 &  2.02 &  6.16926e-02 &  1.01 &  5.71603e-02 &  1.01\\
    \bottomrule
  \end{tabular}}
\label{Ex2_P0P1poly}
\end{table}

\begin{table}[!ht]
  \centering
  \caption{Example 2 (Polygon mesh):
    Error and the convergence rate for $k=1$.}
    \renewcommand{\arraystretch}{1.125}
    \resizebox{11.5cm}{!}{
  \begin{tabular}{@{}c c c c c c c c @{}}
    \toprule
  $ h $   &    $ N_K$ & $ \|\boldsymbol{u}-\boldsymbol{u}_0\| $  
  &  Rate   &   $ \|\boldsymbol{u}-\boldsymbol{u}_h\|_{1,h} $  &  Rate &
  $ \|\boldsymbol{\sigma}-\boldsymbol{\sigma}_h\|_{0,h} $ & Rate
    \\
    \hline
1.250e-01 &     64 &  1.43011e-03 &  -- &  6.57950e-02 &  -- &  3.66878e-02 &  --\\
6.250e-02 &    256 &  1.52928e-04 &  3.23 &  1.45689e-02 &  2.18 &  7.99437e-03 &  2.20\\
3.125e-02 &   1024 &  1.90678e-05 &  3.00 &  3.66617e-03 &  1.99 &  2.03719e-03 &  1.97\\
1.562e-02 &   4096 &  2.29156e-06 &  3.06 &  8.95349e-04 &  2.03 &  4.92205e-04 &  2.05\\
7.069e-03 &  20014 &  2.84402e-07 &  3.01 &  2.22949e-04 &  2.01 &  1.21800e-04 &  2.01\\
    \bottomrule
  \end{tabular}}
\label{Ex2_P1P2poly}
\end{table}

\section{Conclusion}\label{sec:conclusion}

In this paper, we develop and analyze a SDG method for linear elasticity problems. The formulation is derived from the Hellinger-Reissner variational principle, which enables the simultaneous approximation of both the stress tensor and the displacement field. One of the central contribution of this work is the construction of strongly symmetric stress tensors that enforce normal continuity across element boundaries on general polytopal meshes. The careful design of the DoFs for the stress approximation plays a crucial role in achieving this structure.

To achieve this structure, the stress tensor is represented as a piecewise discontinuous polynomial within each polygonal element. The DoFs for the stress are defined without any vertex-based unknowns, which significantly simplifies the implementation after using the hybridization techniques. The use of hybridization not only reduces programming complexity but also, when combined with a Schur complement approach, leads to a final stiffness matrix that involves only DoFs associated with element boundaries. This can substantially decrease the computational cost.

\bibliographystyle{siamplain}
 \bibliography{Ela_SDG}

\end{document}